\documentclass[number, 12pt]{elsarticle}

\usepackage{amsmath}
\usepackage{amsthm}
\usepackage{amssymb}
\usepackage{amsfonts}
\usepackage{dsfont}

\newtheorem{theorem}{Theorem}[section]
\newtheorem{corollary}[theorem]{Corollary}
\newtheorem{proposition}[theorem]{Proposition}
\newtheorem{lemma}[theorem]{Lemma}

\newtheorem{remark}[theorem]{Remark}

\numberwithin{equation}{section}

\def\C{{\operatorname{Cov}}}
\def\E{{\mathds E}}
\def\I{{\mathds 1}}
\def\N{{\mathds N}}
\def\P{{\mathds P}}
\def\R{{\mathds R}}
\def\Sd{{\mathds S}}
\def\V{{\operatorname{Var}}}

\def\bN{{\mathbf N}}

\def\cF{{\mathcal F}}

\def\cH{{\mathcal H}}

\def\cN{{\mathcal N}}

\def\d{\delta}
\def\e{\varepsilon}

\def\k{{\kappa}}

\def\s{\sigma}

\def\dint{{\rm d}}
\def\exp#1{\,{\rm exp} \left(#1\right)}

\begin{document}

\title{\bfseries Limit theory for the Gilbert graph}

\author[uos]{Matthias Reitzner}\ead{matthias.reitzner@uos.de}
\author[unibe]{Matthias Schulte}\ead{matthias.schulte@stat.unibe.ch}
\author[rub]{Christoph Th\"ale}\ead{christoph.thaele@rub.de}

\address[uos]{Institute of Mathematics, Osnabr\"uck University, Germany.}
\address[unibe]{Institute of Mathematical Statistics and Actuarial Science, \\ University of Bern, Switzerland.}
\address[rub]{Faculty of Mathematics, Ruhr University Bochum, Germany.}

\begin{abstract}
For a given homogeneous Poisson point process in $\R^d$ two points are connected by an edge if their distance is bounded by a prescribed distance parameter. The behaviour of the resulting random graph, the Gilbert graph or random geometric graph, is investigated as the intensity of the Poisson point process is increased and the distance parameter goes to zero. The asymptotic expectation and covariance structure of a class of length-power functionals are computed. Distributional limit theorems are derived that have a Gaussian, a stable or a compound Poisson limiting distribution. Finally, concentration inequalities are provided using a concentration inequality for the convex distance.
\end{abstract}


\maketitle


\section{Introduction}

Let $\eta_t$ be a homogeneous Poisson point process of intensity $t >0$ in a compact convex observation window $W\subset\R^d$ with volume $V(W)>0$ and let $(\d_t:t>0)$ be a sequence of positive real numbers such that $\d_t\to 0$, as $t\to\infty$.  A random graph $G(\eta_t, \delta_t)$ is defined by taking the points of $\eta_t$ as its vertices and by connecting two distinct points $x,y\in\eta_t$ by an edge if and only if
$$0<\|x-y\| \leq \delta_t \,, $$
where $\|x-y\|$ stands for the Euclidean distance between $x$ and $y$.
The resulting graph $G(\eta_t,\delta_t)$ is called Gilbert graph, random geometric graph or distance graph, and in the special cases $d=1$ and $d=2$ also interval or disc graph, respectively. It has been introduced (in the planar case) by Gilbert \cite{Gilbert} in 1961. As opposed to the Erd\H{o}s-R\'enyi random graph, which is a purely combinatorial object, the Gilbert graph is a random \textit{geometric} graph because in its construction the relative position of the points in space plays an essential r\^ole.

The aim of the present paper is to investigate functionals related to the edge lengths of the Gilbert graph. The length-power functionals $L^{(\tau)}_t$ of interest are defined by
\begin{equation}\label{def:Lalphat}
L^{(\tau)}_t:=\frac{1}{2}\sum_{(x,y)\in \eta^2_{t,\neq}}\I(\| x-y\| \leq \d_t) \, \| x - y \|^\tau\,,
\end{equation}
where $\tau\in\R$ and $\eta^2_{t,\neq}$ stands for the set of all pairs of distinct points of $\eta_t$. The cases $\tau=0$ and $\tau=1$ are of particular importance. Namely, $L^{(0)}_t$ is the number of edges of $G(\eta_t,\delta_t)$ and $L_t^{(1)}$ is its total edge length. In our analysis, we focus on the asymptotic behaviour of $L_t^{(\tau)}$, as $t\to\infty$ and $\delta_t\to 0$. For this situation we first compute the asymptotic expectation of $L^{(\tau)}_t$ and thereby establish a connection to the covariogram of the underlying convex set $W$. We also analyze the asymptotic covariances of $L^{(\tau_1)}_t$ and $L^{(\tau_2)}_t$ for $\tau_1,\tau_2>-d/2$. In a next step, we develop an understanding for the asymptotic behaviour of the length-powers of the individual edges. In this context, we will show that the collection of all edge length-powers converges, after a suitable re-scaling, to a Poisson point process on the real line. We then develop a comprehensive distributional limit theory for the functionals $L_t^{(\tau)}$ for all powers $\tau$. Depending on the choice of $\tau$ and the distance parameters $(\d_t:t>0)$, we obtain central limit theorems as well as non-central limit theorems in which a stable or a compound Poisson limiting random variable shows up. These results provide a complete picture of the asymptotic distributional behaviour of the length-power functionals $L_t^{(\tau)}$ and substantially add to the existing literature. Moreover, we shall also provide multivariate versions of the mentioned central and non-central limit theorems. Our main tool to prove the central limit theorems is the recently developed Malliavin-Stein method for Poisson functionals (see \cite{PSTU10,PeccatiZheng10,SchulteKolmogorov}), while the proofs of the point process convergence result and the non-central limit theorems rest upon recent findings in \cite{DST,STScaling}. Our investigations of the behaviour of $L_t^{(\tau)}$ are completed by concentration inequalities, which are based on Talagrand's convex distance and its relative for Poisson point processes
introduced in \cite{Reitzner13}.

For the asymptotic behaviour of the Gilbert graph the interplay between the intensity $t$ and the distance parameter $\delta_t$ plays a crucial r\^ole. Clearly, the number of vertices of $G(\eta_t,\delta_t)$ is just the cardinality of $\eta_t$, which is a Poisson random variable with expectation $t\,V(W)$ by the definition of a Poisson point process. In addition, the expected number of edges satisfies the approximation
$$
\E L^{(0)}_t \approx \frac{\kappa_d}2\,  t^2 \d_t^d \, V (W) \qquad{\rm as}\qquad t\to \infty\,,
$$
where $\kappa_d$ stands for the volume of the $d$-dimensional unit ball and $\approx$ means that the quotient of the left and the right hand side tends to $1$, as $t\to\infty$, see Section~\ref{sec:expectationAndCovariance} for details. Heuristically, this means that the degree of a typical vertex, i.e., the number of edges emanating from this vertex, is approximately of order $\kappa_d t \d_t^d$ (this can be made precise using the concept of Palm distributions). This heuristic observation about the degree of a typical vertex naturally leads to three different asymptotic regimes as introduced in Penrose's book~\cite{Penrose03}. These are
\begin{itemize}
\item the {\it sparse regime}, where we assume that
$\lim\limits_{t\to\infty}t \, \d_t^d=0$, implying that the degree of a typical vertex tends to zero,
\item the {\it thermodynamic regime}, where we assume that
$\lim\limits_{t\to\infty}t \, \d_t^d=c\in (0,\infty)$, implying that the degree of a typical vertex is asymptotically constant,
\item the {\it dense regime}, where we have
$ \lim\limits_{t\to\infty}t \, \d_t^d=\infty $,
which means that the degree of a typical vertex of the Gilbert graph tends to infinity.
\end{itemize}

\medskip
There is a vast literature on the Gilbert graph. First to note is Penrose's research monograph \cite{Penrose03}, which also summarizes the developments until 2003. More recent findings which are relevant in our context are -- among others -- due to Bourguin and Peccati \cite{BourguinPeccati2012}, Decreusefond, Schulte and Th\"ale \cite{DST}, Lachi\'eze-Rey and Peccati \cite{LachiezeReyPeccatiI,LachiezeReyPeccatiII} and Reitzner and Schulte \cite{ReitznerSchulte2011} or Schulte and Th\"ale \cite{STScaling}. Important investigations not touched in this paper concern subgraph counting statistics (a far reaching generalization of the concept of $L^{(0)}_t$) and percolation, which are at the core of Penrose's book.

\medskip
This paper is organized as follows. In Section \ref{sec:Preliminaries} we fix some general notation and recall the definition and some important properties of Poisson point processes. Asymptotic expectations as well as asymptotic variances and covariances for the functionals $L_t^{(\tau)}$ are derived in Section~\ref{sec:expectationAndCovariance}, while Section~\ref{sec:PointProcessConvergence} is concerned with the point process of $\tau$-powers of the edge lengths. The distributional limit theory for $L_t^{(\tau)}$ is the content of Section~\ref{sec:DistributionalLimits} and Section~\ref{sec:LDI} contains concentration inequalities for $L_t^{(\tau)}$ in the Poisson and the binomial case.

\section{Preliminaries}\label{sec:Preliminaries}

\paragraph{General notation} In this paper we frequently use the following notation. By $(\Omega,\cF,\P)$ we mean our underlying probability space; expectation, variance and covariance of random variables $X$ and $Y$ with respect to $\P$ are denoted by $\E X$, $\V X$ and $\C(X,Y)$, respectively. We also write $\I(\,\cdot\,)$ for an indicator function.

We let $\lambda$ stand for the Lebesgue measure on $\R^d$, where $d\geq 1$ is a fixed integer. For a compact and convex set $W\subset\R^d$, $V(W):=\lambda(W)$ and $S(W)$ are the volume and the surface area of $W$, respectively. A $d$-dimensional ball with centre $x\in\R^d$ and radius $r>0$ is denoted by $B^d(x,r)$ and for a non-negative integer $j$, $\kappa_j$ stands for the volume of the $j$-dimensional unit ball $B^j(0,1)$. The unit sphere in $\R^d$ is denoted by ${\Sd}^{d-1}$.

We also use the Landau notation. That is, for $g,h:\R\to\R$ we write $g=o(h)$ if $\lim\limits_{t\to\infty}|g(t)|/|h(t)|=0$, $g=O(h)$ if $\limsup\limits_{t\to\infty}|g(t)|/|h(t)|=c\in\R$ and $g=\Theta(h)$ if $g=O(h)$ and $h=O(g)$.

\paragraph{Poisson point processes} Let $\bN(W)$ be the space of finite (simple) counting measures $\eta= \sum_{i=1}^n \e_{x_i}$, where $x_1,\ldots,x_n\in W$, $n\in\N$, are distinct points and where $\e_x$ stands for the unit-mass Dirac measure concentrated at $x\in W$. Alternatively, one can think of $\bN(W)$ as the set of all finite point configurations of distinct points from $W$. This can be achieved by identifying the measure $\eta$ with its support, which forms a closed subset of $W$, cf.\ \cite[Lemma 3.1.4]{SW}. For $\eta \in \bN(W)$ and a Borel set $A \subset \R^d$, $\eta(A)$ is the number of points of $\eta$ falling in $A$ and $\eta\cap A$ stands for the restricted point configuration  $\{ x_1, \dots , x_n\} \cap A$. Due to the geometric flavour of the Gilbert graph, in most cases we will think of $\eta$ as a set or configuration of points in $W$. The space $\bN(W)$ is endowed with the $\sigma$-field $\cN(W)$ generated by the evaluation mappings $T_A: \bN(W)\to\R, \eta\mapsto\eta(A)$ for Borel sets $A\subset W$, see \cite[Chapter 3.1]{SW}.

A random counting measure $\eta_t$, i.e., a random variable defined on the probability space $(\Omega,\cF,\P)$ with values in $(\bN(W),\cN(W))$, is called a (homogeneous) Poisson point process in $W$ with intensity $t>0$ if $\P(\eta_t(A)=0)={\rm exp}(-t\lambda(A))$ for any Borel set $A\subset W$. It should be noted that under these circumstances R\'enyi's theorem (see Section 3.4 in \cite{Kingman}) implies that $\eta_t(A)$ is Poisson distributed with mean $t\lambda(A)$ and that for disjoint Borel sets $A_1,\ldots,A_m\subset W$, $m\in\N$, the random variables $\eta_t(A_1),\ldots,\eta_t(A_m)$ are independent. Alternatively, one can think of $\eta_t$ as a random set of $\eta_t(W)$ random points, which are independently placed within $W$ according to the uniform distribution.

\paragraph{Multivariate Mecke formula} One of the main tools of our analysis is the multivariate Mecke formula for Poisson point processes. In our set-up, it says that
\begin{equation}\label{eq:Mecke}
\begin{split}
& \E\sum_{(x_1,\ldots,x_k)\in\eta^k_{t,\neq}}  f(x_1,\ldots,x_k,\eta_t)\\
& = t^k\int\limits_{W^k} \E f\Big(x_1,\ldots,x_k,\eta_t+\sum_{i=1}^k\e_{x_i}\Big)\,\dint (x_1,\ldots ,x_k)\,,
\end{split}
\end{equation}
where $k\geq 1$ is a fixed integer, $f:W^k\times \bN(W)\to\R$ is a non-negative measurable function and $\eta_{t,\neq}^k$ is the set of all $k$-tuples of distinct points of $\eta_t$, cf. \cite[Corollary 3.2.3]{SW}. If $f$ is only a function on $W^k$ and does not depend on $\eta_t$, which will often be the case in the sequel, the expectation on the right-hand side can be omitted.

\section{Expectation and covariance structure}\label{sec:expectationAndCovariance}

We begin by investigating the expectation of $L_t^{(\tau)}$, recall \eqref{def:Lalphat} for the definition. Let $g_W(y)=V(W\cap(W+y))$, $y\in\R^d$, be the so-called covariogram of $W$. This functional is well-known in convex geometry and has a long history. In particular, we refer to the recent breakthrough by Averkov and Bianchi \cite{MateronConjecture} regarding the famous covariogram problem, the work of Galerne \cite{Gal11}, and the references cited therein.
The following result gives a connection between $\E L_t^{(\tau)}$ and the covariogram of $W$, which seems not to have been noticed so far.

\begin{theorem}\label{thm:Expectation}
If $\tau >-d$, one has that
\begin{equation}\label{eq:ExpansionEL}
\E L^{(\tau)}_t=\frac {t^2}2  \int\limits_{B^d(0, \d_t)} \|y\|^\tau g_W(y)\, \dint y
\end{equation}
and
\begin{equation}\label{eq:BoundsEL}
0 \leq
\frac {d\kappa_d}{2(\tau+d)}  V(W)
- \frac{  \E L^{(\tau)}_t}{t^{2} \d_t^{\tau+d}}
\leq
\frac {\kappa_{d-1}}{2(\tau+d+1)} \d_t S(W)
\,.
\end{equation}
\end{theorem}

\begin{remark}\label{rem:PreciseExpectationBound}\rm
Theorem~\ref{thm:Expectation} especially shows that the number of edges of the Gilbert graph is of order $t^2\d_t^d$, whereas its total edge length is of order $ t^2 \d_t^{d+1}$.
\end{remark}

\begin{proof}[Proof of Theorem~\ref{thm:Expectation}]
We apply the multivariate Mecke formula \eqref{eq:Mecke} with $k=2$ and $f(x,y)=\I(  \| x - y \| \leq {\d_t} ) \| x-y\|^\tau$ to obtain
\begin{eqnarray*}
\E L^{(\tau)}_t
&=&  \frac {t^2}2  \int \limits_{W^2}\I(  \| x-y \| \leq {\d_t} ) \| x-y\|^\tau \, \dint (x , y)\\
&=& \frac {t^2}2  \int \limits_{\R^d} \I(  \| y \| \leq {\d_t} ) \|y\|^\tau \Bigg( \ \int \limits_{\R^d} \I(x \in W,\,x-y \in W) \, \dint x \Bigg) \dint y\\
&=&
\frac{t^2}2 \int \limits_{B^d(0, \d_t)} \|y\|^\tau g_W(y)\, \dint y\,,
\end{eqnarray*}
which gives \eqref{eq:ExpansionEL}.
Transformation into spherical coordinates yields
\begin{equation}\label{eq:intcov}
\E L^{(\tau)}_t =
\frac {t^2}2  \int \limits_{B^d(0, \d_t)} \|y\|^\tau\, g_W(y)\, \dint y=
\frac {t^2}2  \int \limits_0^{\d_t} r^{\tau+d-1} \int \limits_{{\Sd}^{d-1}} g_W(ru)\, \dint u \, \dint r\,,
\end{equation}
where $\dint u$ stands for the infinitesimal element of the spherical Lebesgue measure.
\iffalse
Now, fix $u\in{\Sd}^{d-1}$ and make a Taylor expansion of $g_W(ru)$ at $r=0$. This gives
\begin{equation}\label{eq:TaylorCovariogram}
g_W(ru)=g_W(0)+r\,\frac{\dint }{ \dint r}g_W(ru)\Big|_{r=0}+o(r)
\end{equation}
for any $r>0$. According to \cite[Theorem 14 (ii)]{Gal11}, the derivative of $g_W(ru)$ at $r=0$ exists and is finite. By integration of \eqref{eq:TaylorCovariogram} with respect to $u$, the facts that $g_W(0)=V(W)$, and that
$$\int\limits_{{\Sd}^{d-1}}\, \frac{\dint }{ \dint r} g_W(ru)\Big|_{r=0}\,\dint u=-\k_{d-1}\,S(W)$$
by \cite[Theorem 14 (iii)]{Gal11}, we infer the relation
\begin{equation}\label{eq:2}
\int\limits_{{\Sd}^{d-1}} g_W(ru)\, \dint u = d \kappa_d  V(W) - \kappa_{d-1} S(W) r + o(r)\, ,\  \text{ for } r>0\,,
\end{equation}
which remains valid for all so-called sets $W$ of finite perimeter. Furthermore,
\fi
Galerne \cite[Theorem 13 (iii)]{Gal11} showed that for given $u\in{\Sd}^{d-1}$, $g_W(ru)$ is a Lipschitz function in $r$ whose Lipschitz constant coincides with the $(d-1)$-dimensional volume $V_{d-1}(W|u^\bot)$ of the orthogonal projection of $W$ onto the hyperplane $u^\bot$ orthogonal to $u$. In particular, this implies that $V(W) \geq  g_W(ru) \geq  V(W) - V_{d-1} ( W \vert u^\perp)  r $ for all $r>0$. Thus, by the well-known Cauchy's surface area formula from integral geometry \cite[Equation (6.12)]{SW} we obtain
\begin{equation}\label{eq:Galbounded}
 d \kappa_d  V(W) \geq
\int\limits_{{\Sd}^{d-1}} g_W(ru)\,\dint u \geq
d \kappa_d V(W) - \kappa_{d-1} S(W) r\,.
\end{equation}
Now, \eqref{eq:BoundsEL} is immediate from \eqref{eq:intcov}  and \eqref{eq:Galbounded}.
\end{proof}

After having investigated the first-moment behaviour of $L_t^{(\tau)}$, we now investigate the covariance structure of these functionals for different values of $\tau>-d$. For $\tau_i, \tau_j > -d$ define
\begin{equation*}
\sigma^{(1)}_{\tau_i \tau_j}:= \begin{cases}
\frac{d\kappa_d}{2|\tau_i+\tau_j+d|} & : \tau_i+\tau_j\neq -d\\ \frac{d\kappa_d}{2} & : \tau_i+\tau_j=-d \end{cases}
\quad \text{and} \quad
\sigma^{(2)}_{\tau_i\tau_j}:=\frac{d^2\kappa_d^2}{(\tau_i+d)(\tau_j+d)}\,.
\end{equation*}
To the best of our knowledge, the structure of the covariance matrix given by \eqref{eq:AsymptoticCovariance} below seems to be new.

\begin{theorem}\label{thm:AsymptoticCovariance}
For $\tau_1,\tau_2 > -d$ such that $\tau_1+\tau_2>-d$ one has the inequality
\begin{equation}\label{eq:Covariance}
0 \leq \, V(W) -
\frac{\C(L_t^{(\tau_1)},L_t^{(\tau_2)}) }{\s^{(1)}_{\tau_1 \tau_2} \, t^2 \, \delta_t^{\tau_1+\tau_2+d} + \s^{(2)}_{\tau_1 \tau_2} \, t^3 \, \delta_t^{\tau_1+\tau_2+2d}} \leq  \delta_t S(W)\,.
\end{equation}
In particular, for $\tau_1>-d/2$ one has the variance asymptotics
$$
 \V L_t^{(\tau_1)}  =  \left( \s^{(1)}_{\tau_1 \tau_1} t^2 \, \delta_t^{2\tau_1+d} + \s^{(2)}_{\tau_1 \tau_1} t^3 \, \delta_t^{2\tau_1+2d} \right)  V(W) (1 + O(\d_t))\,.
$$
Define $\widetilde{L}_t^{(\tau_i)}=(L_t^{(\tau_i)}-\E L_t^{(\tau_i)})/\max\{t \, \delta_t^{\tau_i+d/2}, t^{3/2} \, \delta_t^{\tau_i+d}\}$ with $\tau_i > -  d/2$ for $i=1,\ldots,m$. Then the random vector $(\widetilde{L}_t^{(\tau_1)},\ldots, \widetilde{L}_t^{(\tau_m)})$ has the asymptotic covariance matrix
\begin{equation}\label{eq:AsymptoticCovariance}
\Sigma:=\begin{cases}
 \Sigma^{(1)}  &: \lim\limits_{t\to\infty}t \, \delta_t^d=0\\
\Sigma^{(1)}+c \, \Sigma^{(2)} &: \lim\limits_{t\to\infty}t \, \delta_t^d=c\in (0,1]\\
\frac{1}{c} \, \Sigma^{(1)} + \Sigma^{(2)} &: \lim\limits_{t\to\infty}t \, \delta_t^d=c\in (1,\infty)\\
\Sigma^{(2)} &: \lim\limits_{t\to\infty}t \, \delta_t^d=\infty\,,\\
\end{cases}
\end{equation}
with the matrices $\Sigma^{(1)}$ and $\Sigma^{(2)}$ defined as
$$\Sigma^{(1)}:=V(W) \left( \s^{(1)}_{\tau_i \tau_j}\right)_{i,j=1}^m
 \quad \text{and} \quad
\Sigma^{(2)}:= V(W) \left( \s^{(2)}_{\tau_i \tau_j}\right)_{i,j=1}^m\,.$$
\end{theorem}

\begin{proof}
By definition, we have that the product $L_t^{(\tau_1)} L_t^{(\tau_2)} $ equals
$$
\frac{1}{4}  \sum_{(x_1,y_1)\in\eta^2_{t,\neq}} \I(\|x_1-y_1\|\leq \delta_t) \, \|x_1-y_1\|^{\tau_1}   \sum_{(x_2,y_2)\in\eta^2_{t,\neq}} \I(\|x_2-y_2\|\leq \delta_t) \, \|x_2-y_2\|^{\tau_2} \,.$$
We have to distinguish three cases. The first case arises if the points of the two pairs $(x_1,y_1)$ and $(x_2,y_2)$ are all distinct. The second case arises if exactly one of the points of the first pair is identical with one of the points in the second pair. Finally, in the third case both pairs are comprised of the same points of $\eta_t$. Taking additionally into account multiple counting and applying the multivariate Mecke formula \eqref{eq:Mecke} to each of the three resulting sums yields that $\E L_t^{(\tau_1)} L_t^{(\tau_2)}$ equals
\begin{equation*}
\begin{split}
& \frac{t^4}{4}  \int\limits_{W^4}\I(\|x_1-x_2\|,\|x_3-x_4\|\leq\delta_t) \, \|x_1-x_2\|^{\tau_1} \|x_3-x_4\|^{\tau_2}\, \dint (x_1, \dots, x_4)\\
& +t^3 \int\limits_{W^3}\I(\|x_1-x_2\|,  \|x_1-x_3\|\leq\delta_t) \, \|x_1-x_2\|^{\tau_1} \|x_1-x_3\|^{\tau_2}\, \dint (x_1, x_2 ,x_3) \\
&+\frac{t^2}{2} \int\limits_{W^2}\I(\|x_1-x_2\|\leq\delta_t) \, \|x_1-x_2\|^{\tau_1+\tau_2} \, \dint (x_1,x_2) \,.
\end{split}
\end{equation*}
The first term is just the product of $\E L_t^{(\tau_1)}$ and $\E L_t^{(\tau_2)} $ as is evident from the proof of Theorem~\ref{thm:Expectation}, and we see that $\C(L_t^{(\tau_1)},L_t^{(\tau_2)}) $ equals
\begin{equation*}
\begin{split}
 & t^3 \int\limits_{W} \int\limits_W \I(\|y-x_1\|\leq\delta_t) \, \|y-x_1\|^{\tau_1} \, \dint x_1 \int\limits_W \I(\|y-x_2\|\leq\delta_t) \, \|y-x_2\|^{\tau_2} \, \dint x_2 \, \dint y \\
&+\frac{t^2}{2} \int\limits_W \int\limits_W\I(\|x-y\|\leq\delta_t) \, \|x-y\|^{\tau_1+\tau_2} \,\dint x \, \dint y\,.
\end{split}
\end{equation*}
Denote by $W_{-\delta_t}=\{w\in W : B^d(w,\delta_t)\subset W\}$ the (possibly empty) inner parallel set of $W$ and let $\gamma>-d$.
For a point $y\in W_{-\delta_t}$, we see by transforming into spherical coordinates that
$$\int\limits_W \I(\|y-z\|\leq\delta_t) \, \|y-z\|^\gamma \, \dint z
= \int\limits_{B^d(0,\delta_t)}  \I(\|x\|\leq\delta_t) \, \|x\|^\gamma \, \dint x
=\frac{d \, \kappa_d}{\gamma+d} \, \delta_t^{\gamma+d}.$$
For points $y\in W \setminus W_{- \delta_t}$, one has the inequality
$$0\leq \int\limits_W \I(\|y-z\|\leq\delta_t) \, \|y-z\|^\gamma \, \dint z
\leq \frac{d \, \kappa_d}{\gamma+d} \, \delta_t^{\gamma+d} .$$
Observe that
\begin{equation}\label{eq:innpar}
 V(W_{-\delta_t})\geq  V(W)-S(W) \, \delta_t \,.
\end{equation}
Now, we obtain
\begin{equation*}
\begin{split}
& \left( \frac{d \, \kappa_d}{2 \, (\tau_1+\tau_2+d)} \,  t^2 \, \delta_t^{\tau_1+\tau_2+d} + \frac{d^2 \, \kappa_d^2}{(\tau_1+d) \,(\tau_2+d)} \, t^3 \, \delta_t^{\tau_1+\tau_2+2d} \right)  V(W_{-\delta_t})\\
& \qquad \leq \C(L_t^{(\tau_1)},L_t^{(\tau_2)}) \leq  \\
&\hskip1.6cm \left( \frac{d \, \kappa_d}{2 \, (\tau_1+\tau_2+d)} \, t^2 \, \delta_t^{\tau_1+\tau_2+d} + \frac{d^2 \, \kappa_d^2}{(\tau_1+d) \, (\tau_2+d)} \, t^3 \, \delta_t^{\tau_1+\tau_2+2d} \right)  V(W)\,,
\end{split}
\end{equation*}
which together with (\ref{eq:innpar}) yields \eqref{eq:Covariance}. The form of the asymptotic covariance matrix in \eqref{eq:AsymptoticCovariance} is a direct consequence.
\end{proof}

Next, we discuss the definiteness property of the asymptotic covariance matrix $\Sigma$, which has been defined in Theorem~\ref{thm:AsymptoticCovariance}. It is remarkable that this property undergoes a phase transition when moving from the sparse and the thermodynamic regime to the dense one, a phenomenon that has not found attention in the existing literature.

\begin{proposition}\label{prop:definiteness}
For distinct $\tau_i> -d/2$ for $i=1, \dots, m$ and $m\geq 2$, the asymptotic covariance matrix $\Sigma$ given in \eqref{eq:AsymptoticCovariance} is positive definite in the sparse and in the thermodynamic regime, while it is singular in the dense regime.
\end{proposition}

\begin{proof}
The matrix $\Sigma^{(2)}$ is only of rank $1$ so that \eqref{eq:AsymptoticCovariance} implies that the asymptotic covariance matrix is singular in the dense regime for $m\geq 2$. It remains to prove that $\Sigma^{(1)}$ is positive definite. The matrix $\frac{2}{d\kappa_d V(W)}\Sigma^{(1)}$ and all its leading principal minors $\big(\frac{2}{d\kappa_d}\sigma_{\tau_i\tau_j}^{(1)}\big)_{i,j=1}^k$, $k\in\{1,\hdots,m\}$, are Cauchy matrices having determinant
$$ \prod_{1\leq i<j\leq k} \left( \frac {\tau_i-\tau_j}{\tau_i+\tau_j +d} \right)^2 \ >0\,, $$
which proves by Sylvester's criterion the positive definiteness.
\end{proof}

As it can be seen from the proof of Theorem \ref{thm:AsymptoticCovariance}, $\C(L_t^{(\tau_1)},L_t^{(\tau_2)})$ is not well-defined if $\tau_1\leq -d$, $\tau_2\leq -d$ or $\tau_1+\tau_2\leq -d$. To overcome this difficulty we consider the family of truncated length-power functionals
\begin{equation}\label{eqn:DefinitionUta}
U_{t,a}^{(\tau)} :=  \frac{1 }{2}\sum_{(x,y)\in\eta_{t,\neq}^2}\I(t^{-2/d}a\leq\|x-y\|\leq\d_t)\,\|x-y\|^\tau
\end{equation}
for $a\geq0$, $\tau\in\R$ and $t>0$. For $a>0$, $U_{t,a}^{(\tau)}$ has finite moments of all orders. To describe the covariance structure of these truncated functionals let us define
\begin{equation}\label{eqn:DefinitionSigmaTau}
\varrho_\tau(t):=\begin{cases} \max\{t\delta_t^{\tau+d/2},t^{3/2}\delta_t^{\tau+d}\} &:  \tau>-d/2\\
\max\{ t\sqrt{\ln(t^{2/d}\delta_t)},t^{3/2}\delta_t^{d/2}\} &: \tau=-d/2\\
t^{3/2}\delta_t^{\tau+d} &: \tau<-d/2\,. \end{cases}
\end{equation}
After rescaling by $\varrho_\tau(t)$ the truncated functionals $U^{(\tau)}_{t,a}$ have the following asymptotic covariances.

\begin{theorem}\label{thm:CovariancesUta}
Let $a>0$.
\begin{itemize}
\item[(a)] Assume that $t^2\delta_t^d\to\infty$, as $t\to\infty$. For $\tau_i, \tau_j>-d/2$,
$$
\lim_{t\to\infty} \frac{\C\Big(\frac{U_{t,a}^{(\tau_i)}}{\varrho_{\tau_i}(t)},\frac{U_{t,a}^{(\tau_j)}}{\varrho_{\tau_j}(t)}\Big)}{V(W)}= \begin{cases}
 \sigma^{(1)}_{\tau_i\tau_j}  &: \lim\limits_{t\to\infty}t \, \delta_t^d=0\\
\sigma^{(1)}_{\tau_i\tau_j}+c \, \sigma^{(2)}_{\tau_i\tau_j} &: \lim\limits_{t\to\infty}t \, \delta_t^d=c\in (0,1]\\
\frac{1}{c} \, \sigma^{(1)}_{\tau_i\tau_j} + \sigma^{(2)}_{\tau_i\tau_j} &: \lim\limits_{t\to\infty}t \, \delta_t^d=c\in (1,\infty)\\
\sigma^{(2)}_{\tau_i\tau_j} &: \lim\limits_{t\to\infty}t \, \delta_t^d=\infty\,,
\end{cases}
$$
for $\tau_i> -d/2$ and $\tau_j=-d/2$,
$$
\lim_{t\to\infty} \frac{\C\Big(\frac{U_{t,a}^{(\tau_i)}}{\varrho_{\tau_i}(t)},\frac{U_{t,a}^{(\tau_j)}}{\varrho_{\tau_j}(t)}\Big)}{V(W)}= \begin{cases}
 0  &: \lim\limits_{t\to\infty}t\delta_t^d/\ln(t^{2/d}\delta_t)=0\\
\sqrt{c} \, \sigma^{(2)}_{\tau_i\tau_j} &: \lim\limits_{t\to\infty}t\delta_t^d/\ln(t^{2/d}\delta_t)=c\in (0,1]\\
\sigma^{(2)}_{\tau_i\tau_j} &: \lim\limits_{t\to\infty}t\delta_t^d/\ln(t^{2/d}\delta_t)\in (1,\infty]\,, 
\end{cases}
$$
and for $\tau_i=-d/2$,
$$
\lim_{t\to\infty} \frac{\V\Big(\frac{U_{t,a}^{(\tau_i)}}{\varrho_{\tau_i}(t)}\Big)}{V(W)}= \begin{cases}
 \sigma^{(1)}_{\tau_i\tau_i}  &: \lim\limits_{t\to\infty}t\delta_t^d/\ln(t^{2/d}\delta_t)=0\\
\sigma^{(1)}_{\tau_i\tau_i}+c \, \sigma^{(2)}_{\tau_i\tau_i} &: \lim\limits_{t\to\infty}t\delta_t^d/\ln(t^{2/d}\delta_t)=c\in (0,1]\\
\frac{1}{c}\sigma^{(1)}_{\tau_i\tau_i}+\sigma^{(2)}_{\tau_i\tau_i} &: \lim\limits_{t\to\infty}t\delta_t^d/\ln(t^{2/d}\delta_t)=c\in (1,\infty)\\
\sigma^{(2)}_{\tau_i\tau_i} &: \lim\limits_{t\to\infty}t\delta_t^d/\ln(t^{2/d}\delta_t)=\infty\,.
\end{cases}
$$
\item [(b)] Assume that $\tau_i\geq \tau_j\in(-d,-d/2)$ and that $t^{3+4\tau_j/d}\delta_t^{2\tau_j+2d}\to\infty$, as $t\to\infty$. Then,
$$
\lim_{t\to\infty} \C\Big(\frac{U_{t,a}^{(\tau_i)}}{\varrho_{\tau_i}(t)},\frac{U_{t,a}^{(\tau_j)}}{\varrho_{\tau_j}(t)}\Big) = \sigma^{(2)}_{\tau_i,\tau_j}V(W)\,.
$$
\end{itemize}
\end{theorem}
\begin{proof}
By the same arguments as in the proof of Theorem~\ref{thm:AsymptoticCovariance} we obtain that
$$
(V(W)-\delta_t S(W)) (J_1+J_2) \leq \C(U_{t,a}^{(\tau_i)},U_{t,a}^{(\tau_j)}) \leq V(W) (J_1+J_2)
$$
with
$$
J_1= \I(t^{-2/d}a\leq \delta_t) \, \frac{d\kappa_d}{2} t^2 \int_{t^{-2/d}a}^{\delta_t} r^{\tau_i+\tau_j+d-1} \, \dint r
$$
and
$$
J_2= \I(t^{-2/d}a\leq \delta_t) \, d^2\kappa_d^2 t^3 \int_{t^{-2/d}a}^{\delta_t} r^{\tau_i+d-1} \, \dint r \ \int_{t^{-2/d}a}^{\delta_t} r^{\tau_j+d-1} \, \dint r\,.
$$
Evaluation of these integrals shows that, for $t^{-2/d}a\leq \delta_t$,
$$
J_1=\begin{cases} \sigma^{(1)}_{\tau_i\tau_j} t^2 (\delta_t^{\tau_i+\tau_j+d} -(t^{-2/d}a)^{\tau_i+\tau_j+d}) & : \tau_i+\tau_j>-d\\
\sigma^{(1)}_{\tau_i\tau_j} t^2 (\ln \delta_t -\ln(t^{-2/d}a)) & : \tau_i+\tau_j=-d\\
 \sigma^{(1)}_{\tau_i\tau_j} t^2 ((t^{-2/d}a)^{\tau_i+\tau_j+d} - \delta_t^{\tau_i+\tau_j+d})  &: \tau_i+\tau_j\in(-2d,-d)\end{cases}
$$
and
$$
J_2= \sigma^{(2)}_{\tau_i\tau_j} t^{3} (\delta_t^{\tau_i+d}-(t^{-2/d}a)^{\tau_i+d}) (\delta_t^{\tau_j+d}-(t^{-2/d}a)^{\tau_j+d})\,.
$$
The assumptions that $\delta_t\to 0$ and $t^2\delta_t^d\to\infty$, as $t\to\infty$, complete the proof of (a). Note that $t^{3+4\tau_j/d}\delta_t^{2\tau_j+2d}\to\infty$, as $t\to\infty$, implies that $J_1/J_2\to 0$ as $t\to\infty$. This yields the statement of (b).
\end{proof}

\begin{remark}\rm
We can use the same arguments as in the proof of Proposition \ref{prop:definiteness} to investigate the positive definiteness of the covariance matrices given in Theorem \ref{thm:CovariancesUta}. For distinct $\tau_1,\hdots,\tau_m\geq -d/2$ the asymptotic covariance matrix obtained in Theorem \ref{thm:CovariancesUta} (a) is positive definite in the sparse and in the thermodynamic regime. In the dense regime it is singular except of the special case that $m=2$, $\tau_1=-d/2$ or $\tau_2=-d/2$ and $\lim\limits_{t\to\infty}t\delta_t^d/\ln(t^{2/d}\delta_t)=c\in\R$. The covariance matrix in Theorem \ref{thm:CovariancesUta} (b) is singular for all $m\geq 2$.
\end{remark}

\section{Point process convergence}\label{sec:PointProcessConvergence}

Recall the definition \eqref{def:Lalphat} of the length-power functional $L_t^{(\tau)}$,
$$
L_t^{(\tau)} = \sum_{(x,y)\in\eta_{t,\neq}^2} \I(\| x-y\| \leq \d_t) \, \| x - y \|^\tau\,.
$$
In this section we investigate the summands, that is, the building blocks of this random sum. The understanding of their joint asymptotic behaviour is the foundation for the non-central limit theorems developed in Section~\ref{sec:DistributionalLimits} below. To this end, define the point process
$$
\xi_t^{(\tau)}:=\frac{1}{2}\sum_{(x,y)\in\eta_{t,\neq}^2} \I(\| x-y\| \leq \d_t) \, \e_{\|x-y\|^\tau}\,
$$
for $\tau\in\R$ with $\tau\neq 0$, where, recall, $\e_x$ stands for the unit-mass Dirac measure at $x$. We exclude the degenerate case $\tau=0$, where all points are concentrated at $1$, in order to obtain a simple point process, that is, a point process without multiple points. Our attention is focussed on the asymptotic distributional behaviour of the rescaled point process $t^{2\tau/d}\xi_t^{(\tau)}$ on $\R_+$. Part (a) of the following theorem deals with the regime in which the expected number of edges tends to infinity, while part (b) focuses on the case in which the number of edges of the Gilbert graph stays asymptotically constant in the mean in that there is a constant $0<c<\infty$ such that $\lim_{t\to\infty} t^2 \, \delta_t^d=c$. The first case is taken from \cite[Theorem 2.4]{STScaling}, while the second case can be obtained by combining Corollary 3.3 and Lemma 7.14 in \cite{DST}. We also mention that the paper \cite{DST} provides rates of convergence measured in a suitable point process distance, which we do not provide here for the sake of brevity.

\begin{theorem}\label{thm:PPCedgesinfinity}
\begin{itemize}
\item[(a)] If $\tau\in\R$ with $\tau\neq 0$ and $t^2 \, \delta_t^d\to\infty$, as $t\to\infty$, the point process $t^{2\tau/d}\xi_t^{(\tau)}$ converges in distribution to a Poisson point process on $\R_+$ with intensity measure
$$\nu(B)=\frac{d\k_d}{2|\tau|} V(W)\int\limits_B u^{(d/\tau)-1}\,\dint u\,, \qquad B\subset\R_+ \text{ Borel}\,.$$

\item[(b)] If $\tau\in\R$ with $\tau\neq0$ and $t^2 \, \delta_t^d\to c\in (0,\infty)$, as $t\to\infty$, the point process $t^{2\tau/d}\xi_t^{(\tau)}$ converges in distribution to a Poisson point process on $\R_+$ with intensity measure
$$\nu(B)= \frac{d\k_d}{2 |\tau|} V(W)\int\limits_{B} \I(u^{d/\tau}\in[0,c]) \, u^{(d/\tau)-1}\,\dint u\,, \qquad B\subset\R_+ \text{ Borel}\,.$$
\end{itemize}
\end{theorem}

\begin{remark}\rm
We notice that if $t^2\d_t^d\to 0$ as $t\to\infty$, the probability that there are no edges tends to one. This is an immediate consequence of Theorem~\ref{thm:Expectation}, and together with Markov's inequality we see that the rate for this convergence is at least $t^2\delta_t^d$. The limiting point process in this case might be interpreted as the empty point process or configuration.
\end{remark}

Let us present a consequence of Theorem~\ref{thm:PPCedgesinfinity}, which is very useful for some of the proofs in Section \ref{sec:DistributionalLimits}. In what follows, we shall use the notation $\overset{d}{\longrightarrow}$ to indicate convergence in distribution.

\begin{theorem}\label{thm:PPPConv}
Let $m\in\N$, $\tau_1,\ldots,\tau_m\in\R$, $a\in\R_+$ with $a\leq \lim\limits_{t\to\infty} t^{2/d}\delta_t$ and define, for $i\in\{1,\ldots,m\}$,
\begin{equation}\label{eqn:DefinitionLta}
L_{t,a}^{(\tau_i)} := \frac 12 \sum_{(x,y)\in\eta_{t,\neq}^2}\I(\|x-y\|\leq\min\{t^{-2/d}a,\d_t\})\,\|x-y\|^{\tau_i}\,.
\end{equation}
Further, let $\zeta$ be a unit-intensity Poisson point process on $\R_+$ and define the random variables
$$
Z_i := \Big(\frac{2}{\kappa_dV(W)}\Big)^{\tau_i/d}\sum_{x\in\zeta\cap[0,\frac{\k_dV(W)}{ 2}a^d]}x^{\tau_i/d}\,,\qquad i\in\{1,\ldots,m\}\,.
$$
Then, $(t^{2\tau_1/d}L_{t,a}^{(\tau_1)},\ldots,t^{2\tau_m/d}L_{t,a}^{(\tau_m)})\overset{d}{\longrightarrow}(Z_1,\ldots,Z_m)$, as $t\to\infty$.
\end{theorem}
\begin{proof}
Let $\xi^{(d)}$ be a Poisson point process on $\R_+$ whose intensity measure coincides with $\frac{\k_d}{ 2}V(W)$ times the Lebesgue measure. From Theorem~\ref{thm:PPCedgesinfinity} in the special case $\tau=d$ there we infer in connection with Corollary 5.5 and Theorem 16.16 in \cite{Kallenberg} and the fact that the class of non-negative continuous functions on $[0,a^d]$ is closed with respect to non-negative linear combinations that
\begin{equation}\label{eq:ConvDistrKallenberg}\
\begin{split}
& \Big(\sum_{z\in t^2\xi_t^{(d)}\cap[0,a^d]}f_1(z),\ldots,\sum_{z\in t^2\xi_t^{(d)}\cap[0,a^d]}f_m(z)\Big)\\
& \overset{d}{\longrightarrow}\Big(\sum_{z\in\xi^{(d)}\cap[0,a^d]}f_1(z),\ldots,\sum_{z\in\xi^{(d)}\cap[0,a^d]}f_m(z)\Big)\,,
\end{split}
\end{equation}
as $t\to\infty$, for all non-negative continuous functions $f_1,\ldots,f_m$ on $[0,a^d]$. Let us define for $n\in\N$ and $i\in\{1,\ldots,m\}$,
$$
g_i(z):=z^{\tau_i/d}\qquad\text{and}\qquad g_{i,n}(z):=\begin{cases} n^{-\tau_i/d} &: z\in[0,1/n]\\ z^{\tau_i/d} &: z\in(1/n,a^d]\,, \end{cases}
$$
and note that  the functions $g_{1,n},\ldots,g_{m,n}:[0,a^d]\to\R$ are non-negative and continuous. We further notice that $t^{2\tau_i/d}L_{t,a}^{(\tau_i)} = \sum_{z\in t^2\xi_t^{(d)}\cap [0,a^d]}g_i(z)$ for all $i\in\{1,\ldots,m\}$. We also observe that
\begin{align*}
&\P\Big(\sum_{z\in t^2\xi_t^{(d)}\cap [0,a^d]}g_i(z)\neq \sum_{z\in t^2\xi_t^{(d)}\cap [0,a^d]}g_{i,n}(z)\quad\text{for some }i\in\{1,\ldots,m\}\Big)\\
& \leq \frac{1}{2} \E \sum_{(x,y)\in\eta_{t,\neq}^2} \I(\|x-y\|^d\leq t^{-2}/n) \leq \frac{\kappa_d V(W)}{2n}\,,
\end{align*}
which tends to zero, as $n\to\infty$, independently of $t$, and that
$$
\lim_{n\to\infty}\P\Big(\sum_{z\in\xi^{(d)}\cap [0,a^d]}g_i(z)\neq \sum_{z\in\xi^{(d)}\cap [0,a^d]}g_{i,n}(z)\quad\text{for some }i\in\{1,\ldots,m\}\Big)=0\,.
$$
This implies that for a bounded continuous function $h:\R^m\to\R$ and $\varepsilon>0$ we can find $n\in\N$ such that
\begin{align*}
&\Big|\E h((t^{2\tau_i/d}L_{t,a}^{(\tau_i)})_{i=1}^m)-\E h\Big(\Big(\sum_{x\in t^2\xi_t^{(d)}\cap [0,a^d]}g_{i,n}(z)\Big)_{i=1}^m \Big)\Big|\leq\frac{\varepsilon}{3}\,, \\
& \Big|\E h\Big(\Big(\sum_{z\in \xi^{(d)}\cap[0,a^d]}g_{i,n}(z)\Big)_{i=1}^m\Big)-\E h\Big(\Big(\sum_{z\in \xi^{(d)}\cap [0,a^d]}g_i(z)\Big)_{i=1}^m\Big)\Big|\leq\frac{\varepsilon}{3}
\end{align*}
for all $t>0$. Moreover, we also have that
$$
\Big|\E h\Big(\Big(\sum_{z\in t^2\xi_t^{(d)}\cap [0,a^d]}g_{i,n}(z)\Big)_{i=1}^m\Big)-\E h\Big(\Big(\sum_{z\in \xi^{(d)}\cap [0,a^d]}g_{i,n}(z)\Big)_{i=1}^m\Big)\Big|\leq\frac{\varepsilon}{3}
$$
for all sufficiently large $t$, by \eqref{eq:ConvDistrKallenberg} and the properties of $g_{1,n},\ldots,g_{m,n}$. Combining these estimates with the triangle inequality, we conclude that
$$
\Big|\E h((t^{2\tau_i/d}L_{t,a}^{(\tau_i)})_{i=1}^m)-\E h\Big(\Big(\sum_{z\in \xi^{(d)}\cap [0,a^d]}g_{i}(z)\Big)_{i=1}^m\Big)\Big|\leq\varepsilon
$$
for sufficiently large $t$. This proves that, as $t\to\infty$,
$$
(t^{2\tau_1/d}L_{t,a}^{(\tau_1)},\ldots,t^{2\tau_m/d}L_{t,a}^{(\tau_m)})\overset{d}{\longrightarrow}\Big(\sum_{z\in \xi^{(d)}\cap [0,a^d]}g_{1}(z),\ldots,\sum_{z\in \xi^{(d)}\cap [0,a^d]}g_{m}(z)\Big)\,,
$$
where by a change of variables the limiting random vector has the same distribution as $(Z_1,\ldots,Z_m)$.
\end{proof}

\section{Distributional limit theorems}\label{sec:DistributionalLimits}

\subsection{Central limit theorems}

The present section is devoted to the question for which choices of $\tau$ and $(\d_t: t>0)$ the suitably normalized length-power functionals $L_t^{(\tau)}$ converge in distribution to a standard Gaussian random variable.

\begin{theorem}\label{thm:CLT}
Assume that $t^2\d_t^d\to\infty$, as $t\to\infty$, and let $N\sim\cN(0,1)$ be a standard Gaussian random variable.
\begin{itemize}
\item[(a)] If $\tau>-d/2$, then
$$
\frac{L_t^{(\tau)}-\E L_t^{(\tau)}}{\sqrt{\V L_t^{(\tau)}}}\overset{d}{\longrightarrow} N, \quad \text{as} \quad t\to\infty.
$$
\item [(b)] If $\tau=-d/2$, then
$$
\frac{L_t^{(\tau)}-\E L_t^{(\tau)}}{\sqrt{V(W)}\sqrt{d\kappa_d t^2 \ln(t^{2/d}\delta_t)/2+4\kappa_d^2t^3\delta_t^d}}\overset{d}{\longrightarrow} N, \quad \text{as} \quad t\to\infty.
$$
\item[(c)] If $\tau\in(-d,-d/2)$ and $t^{3+4\tau/d}\d_t^{2(d+\tau)}\to \infty$, as $t\to\infty$, then
$$
\frac{L_t^{(\tau)}-\E L_t^{(\tau)}}{d\kappa_d \sqrt{V(W)} t^{3/2} \delta_t^{\tau+d}/(\tau+d)}\overset{d}{\longrightarrow} N, \quad \text{as} \quad t\to\infty.
$$
\end{itemize}
\end{theorem}

Theorem~\ref{thm:CLT} follows directly from our more general multivariate central limit theorem that we present next. For that purpose we define $\widetilde{L}_t^{(\tau)}=(L_t^{(\tau)}-\E L_t^{(\tau)})/\varrho_\tau(t)$ for $t>0$ and $\tau>-d$, where $\varrho_\tau(t)$ was defined in \eqref{eqn:DefinitionSigmaTau}.

\begin{theorem}\label{thm:CLTMulti}
Let $\tau_1>\tau_2>\ldots>\tau_m$ for $m\in\N$ and assume that
\begin{itemize}
\item [(a)] $\tau_m\geq -d/2$ and $t^2\d_t^d\to\infty$, or
\item [(b)] $\tau_m\in (-d,-d/2)$ and $t^{3+4\tau_m/d}\d_t^{2(d+\tau_m)}\to\infty$, as $t\to\infty$.
\end{itemize}
Let ${\bf N}_\Sigma$ be an $m$-dimensional centred Gaussian random vector whose covariance matrix $\Sigma$ is the asymptotic covariance matrix of $\big(U_{t,a}^{(\tau_i)}/\varrho_{\tau_i}(t)\big)_{i=1,\hdots,m}$ given in Theorem \ref{thm:CovariancesUta}. Then,
$$
\big(\widetilde{L}_t^{(\tau_1)},\ldots,\widetilde{L}_t^{(\tau_m)}\big)\overset{d}{\longrightarrow} {\bf N}_{\Sigma}\,,\qquad\text{as}\qquad t\to\infty\,.
$$
\end{theorem}

Finally, we present a quantitative version of Theorem~\ref{thm:CLT} (a) in the partial regime $\tau> -d/4$. Its proof and that of Theorem~\ref{thm:CLTMulti} make use of the so-called Malliavin-Stein method. We emphasize that the main result in \cite{Sunklodas} can be used to show that the rate of convergence provided in this way is in fact optimal up to the numerical constant $C$ (we also refer to the paper \cite{AzmoodehPeccati} in which for the case $\tau=0$ the optimality of the rate for a different probability metric, the so-called Wasserstein distance, has been verified by other methods).

\begin{theorem}\label{thm:CLTRateUniv}
Let $\tau > -d/4 $ and let $N\sim\cN(0,1)$ be a standard Gaussian random variable. Then there is a constant $C>0$ only depending on $\tau$, $W$ and $(\delta_t:t\geq 1)$ such that
$$\sup_{x\in\R}\left|\P\left(\frac{L_t^{{(\tau)}} - \E L_t^{{(\tau)}}}{\sqrt{\V L_t^{{(\tau)}}}}\leq x\right) - \P\left(N\leq x\right)\right|\leq C \, t^{-\frac{1}{2}} \,\max\{1,(t \, \delta_t^d)^{-\frac{1}{2}}\}$$
for $t\geq 1$.
\end{theorem}

The following lemma allows us to approximate $L_t^{(\tau)}$ by its truncated version $U_{t,a}^{(\tau)}$, which was defined in \eqref{eqn:DefinitionUta}.

\iffalse
\begin{lemma}\label{lem:LtUt}
\begin{itemize}
\item [(a)] For $\tau\in\R$ and $a> 0$,
$$
\E|L_t^{(\tau)}-U_{t,a}^{(\tau)}| \leq \frac{d\kappa_d V(W)}{2(\tau+d)}\,\frac{a^{d+\tau}}{t^{2\tau/d}}\,.
$$
\item [(b)] Let $\tau>-d$ and $a> 0$ and suppose that $t^2\delta_t^d\to\infty$, as $t\to\infty$. If $\tau\in(-d,-d/2)$, assume additionally that $t^{3+4\tau/d}\delta_t^{2\tau+2d}\to\infty$, as $t\to\infty$. Then,
$$
\lim_{t\to\infty}\E\Big|  \frac{L_t^{(\tau)}-\E L_t^{(\tau)} }{ \varrho_\tau(t)} - \frac{U_{t,a}^{(\tau)}-\E U_{t,a}^{(\tau)} }{ \varrho_\tau(t)}\Big|=0\,.
$$
\end{itemize}
\end{lemma}
\begin{proof}
By the same arguments as in the proof of Theorem \ref{thm:Expectation} we obtain that
\begin{align*}
\E|L_t^{(\tau)}-U_{t,a}^{(\tau)}| & = \frac{t^2 }{2}\int\limits_{W^2}\I(\|x-y\|\leq\min\{t^{-2/d}a,\d_t\})\,\|x-y\|^\tau\,\dint(x,y)\\
& \leq \frac{t^2 }{2}\frac{d\kappa_d V(W)}{\tau+d}\,(t^{-2/d}a)^{d+\tau} = \frac{d\kappa_d V(W)}{2(\tau+d)}\,\frac{a^{d+\tau}}{t^{2\tau/d}}\,,
\end{align*}
which is (a). For the situation in (b) this inequality implies that 
\begin{align*}
\E\Big| \frac{L_t^{(\tau)}-\E L_t^{(\tau)} }{ \varrho_\tau(t)} -  \frac{U_{t,a}^{(\tau)}-\E U_{t,a}^{(\tau)} }{ \varrho_\tau(t)}\Big| & \leq  \frac{2 }{ \varrho_\tau(t)}\,\E|L_t^{(\tau)}-U_{t,a}^{(\tau)}|\\
& \leq  \frac{d\kappa_d V(W)}{\tau+d}\,\frac{a^{d+\tau}}{t^{2\tau/d}\varrho_\tau(t)}\,.
\end{align*}
Since $t^{2\tau/d}\varrho_\tau(t)\to\infty$ by \eqref{eqn:DefinitionSigmaTau} and our assumptions on $\tau$ and $(\delta_t:t\geq 1)$, the second part of the result follows.
\end{proof}
\fi

\begin{lemma}\label{lem:LtUt}
Let $\tau>-d$ and $a\geq 0$ and suppose that $t^2\delta_t^d\to\infty$, as $t\to\infty$. If $\tau\in(-d,-d/2)$, assume additionally that $t^{3+4\tau/d}\delta_t^{2\tau+2d}\to\infty$, as $t\to\infty$. Then,
$$
\lim_{t\to\infty}\E\Big|  \frac{L_t^{(\tau)}-\E L_t^{(\tau)} }{ \varrho_\tau(t)} - \frac{U_{t,a}^{(\tau)}-\E U_{t,a}^{(\tau)} }{ \varrho_\tau(t)}\Big|=0\,.
$$
\end{lemma}
\begin{proof}
It holds that
\begin{align*}
&\E\Big| \frac{L_t^{(\tau)}-\E L_t^{(\tau)} }{ \varrho_\tau(t)} -  \frac{U_{t,a}^{(\tau)}-\E U_{t,a}^{(\tau)} }{ \varrho_\tau(t)}\Big|\leq  \frac{2 }{ \varrho_\tau(t)}\,\E|L_t^{(\tau)}-U_{t,a}^{(\tau)}|\\
&\leq  \frac{t^2 }{ \varrho_\tau(t)}\int\limits_{W^2}\I(\|x-y\|\leq\min\{t^{-2/d}a,\d_t\})\,\|x-y\|^\tau\,\dint(x,y)\\
& \leq \frac{t^2 }{ \varrho_\tau(t)}\frac{d\kappa_d V(W)}{\tau+d}\,(t^{-2/d}a)^{d+\tau}\,.
\end{align*}
Since $t^{2\tau/d}\varrho_\tau(t)\to\infty$ by \eqref{eqn:DefinitionSigmaTau} and our assumptions on $\tau$ and $(\delta_t:t>0)$, the result follows.
\end{proof}

In the next step, we approximate the random vector
$$
{\bf U}_{t,a}:=\bigg(\frac{U^{(\tau_1)}_{t,a}-\E U^{(\tau_1)}_{t,a}}{\varrho_{\tau_1}(t)},\ldots,\frac{U^{(\tau_m)}_{t,a}-\E U^{(\tau_m)}_{t,a}}{\varrho_{\tau_m}(t)} \bigg)
$$
with $a> 0$ and $\tau_1>\ldots>\tau_m>-d$, $m\in\N$, by a Gaussian random vector. This will be done by using a bound, which was derived by the Malliavin-Stein method. In order to present this result, we need to introduce some more notation.

Let $\mathcal{G}_{\tau_i\tau_j}$ be the set of all connected weighted multigraphs (without loops) $G=(V,E)$ with four edges such that two edges have weight $\tau_i$ and the two remaining edges have weight $\tau_j$. For $e\in E$ we denote by $\tau_e$ the weight of the edge $e$ and by $v_1(e)$ and $v_2(e)$ its two endpoints. For a graph $G\in\mathcal{G}_{\tau_i\tau_j}$ and $a,t>0$ we define
\begin{align*}
M_G(a,t) := t^{|V|} \int_{W^{|V|}} \prod_{e\in E} & \I(t^{-2/d}a\leq \|x_{v_1(e)}-x_{v_2(e)}\|\leq \delta_t)\\ & \times \|x_{v_1(e)}-x_{v_2(e)}\|^{\tau_e}  \, \dint(x_{v})_{v\in V}
\end{align*}
and
$$
M_{\tau_i\tau_j}(a,t) := \max_{G\in\mathcal{G}_{\tau_i\tau_j}} M_G(a,t).
$$
Later we shall use that these terms even exist for $a=0$ if $\tau_i,\tau_j>-d/4$.

In order to compare the distributions of two $m$-dimensional random vectors ${\bf Y},{\bf Z}$ with $\E\|{\bf Y}\|^2<\infty$ and $\E\|{\bf Z}\|^2<\infty$, we will use the so-called $d_3$-distance that is defined as
$$
d_3({\bf Y},{\bf Z}):=\sup_{g\in{\cal H}_m} |\E g({\bf Y})-\E g({\bf Z})|\,.
$$
Here, ${\cal H}_m$ stands for the class of all thrice continuously differentiable functions $g: \R^m\to\R$ such that
$$\max_{1\leq i_1, i_2\leq m}\sup_{x\in\R^m}\left|\frac{\partial^2 g}{\partial x_{i_1} \partial x_{i_2}}(x)\right|\leq 1\,, \quad \max_{1\leq i_1, i_2, i_3\leq m}\sup_{x\in\R^m}\left|\frac{\partial^3 g}{\partial x_{i_1} \partial x_{i_2} \partial x_{i_3}}(x)\right|\leq 1\,.$$ Notice that $\cH_m$ is a  convergence determining class, that is, convergence in $d_3$-distance implies weak convergence of the involved laws.

For the $d_3$-distance between ${\bf U}_{t,a}$ and an $m$-dimensional centred Gaussian random vector ${\bf N}_\Sigma$ with covariance matrix $\Sigma=(\sigma_{ij})_{i,j=1}^m$ we have the bound
\begin{equation}\label{eq:Boundd3}
\begin{split}
d_3({\bf U}_{t,a},{\bf N}_\Sigma) & \leq \frac{1}{2} \sum_{i,j=1}^m \left| \sigma_{ij}-\frac{\C(U_{t,a}^{(\tau_i)},U_{t,a}^{(\tau_j)})}{\varrho_{\tau_i}(t)\varrho_{\tau_j}(t)}\right|\\
& \quad +Cm \left(\sum_{\ell=1}^m \frac{\sqrt{\V U_{t,a}^{(\tau_\ell)}}}{\varrho_{\tau_\ell}(t)}+1\right) \sum_{i,j=1}^m \frac{\sqrt{M_{\tau_i\tau_j}(a,t)}}{\varrho_{\tau_i}(t)\varrho_{\tau_j}(t)}
\end{split}
\end{equation}
with an absolute constant $C>0$, which can be deduced from Theorem 4.2 in \cite{PeccatiZheng10} in a similar way as the bound for the Wasserstein distance in \cite{ReitznerSchulte2011} is obtained from Theorem 3.1 in \cite{PSTU10} (for an exact proof we refer to Theorem 6.3 in the second author's PhD thesis \cite{SchulteDiss}).

For the Kolmogorov distance between the standardization of $U^{(\tau)}_{t,a}$ and a standard Gaussian random variable $N\sim\cN(0,1)$ it follows from Theorem 4.2 in \cite{SchulteKolmogorov} that
\begin{equation}\label{eq:BoundKolmogorov}
\sup_{x\in\R}\left| \P\left(\frac{U_{t,a}^{(\tau)} - \E U_{t,a}^{(\tau)}}{\sqrt{\V U_{t,a}^{(\tau)}}}\leq x \right) -\P\left(N\leq x\right)\right|\leq C \, \frac{\sqrt{M_{\tau\tau}(a,t)}}{\V U_{t,a}^{(\tau)}}
\end{equation}
with an absolute constant $C>0$.

Thus, in order to show Theorem~\ref{thm:CLTMulti} and Theorem~\ref{thm:CLTRateUniv} we need to find upper bounds for $M_{\tau_i\tau_j}(a,t)$ and to control the convergence of the covariance matrix of ${\bf U}_{t,a}$ to $\Sigma$.

\begin{proof}[Proof of Theorem \ref{thm:CLTMulti}]
To bound $M_{\tau_i\tau_j}$, let $\mathcal{G}^T_{\tau_i\tau_j}$ be the set of all $G\in\mathcal{G}_{\tau_i\tau_j}$ that have no circles (we use the convention that two vertices connected by multiple edges do not form a circle). In other words, $\mathcal{G}^T_{\tau_i\tau_j}$ is the subset of trees in $\mathcal{G}_{\tau_i\tau_j}$. Then, it is clear that there exists a constant $C_T>0$ such that
$$
M_{\tau_i\tau_j}(a,t)\leq C_T \max_{G\in \mathcal{G}^T_{\tau_i\tau_j}} M_{G}(a,t)\,.
$$
Indeed, for any point configuration $(x_v)_{v\in V}$ the integrand in the definition of $M_{G}(a,t)$ can be bounded by the integrand of another multigraph which is obtained by shifting weighted edges. For a tree $G\in \mathcal{G}^T_{\tau_i\tau_j}$, we obtain by successively integrating with respect to the variables corresponding to the leaves that
\begin{equation}\label{eqn:BoundMG}
M_{G}(a,t) \leq \frac{V(W)t}{16} \prod_{\substack{\{v_1,v_2\}\subset V\\ (v_1,v_2)\in E}} d\kappa_d t \int_{\min\{t^{-2/d}a,\d_t\}}^{\delta_t} r^{d-1+\sum_{e\in E, e=(v_1,v_2)}\tau_e} \, \dint r\,.
\end{equation}
Note that for any $a>0$ and $\tau,\tau_1,\tau_2>-d$ there are constants $C(\tau)>0$ and $C(a,\tau_1,\tau_2)>0$ such that
\begin{align*}
& \frac{\sqrt{t}}{\varrho_\tau(t)} t \int_{\min\{t^{-2/d}a,\d_t\}}^{\delta_t} r^{d-1+\tau} \, \dint r \leq C(\tau)\,, \allowdisplaybreaks\\
& \frac{t}{\varrho_{\tau_1}(t) \varrho_{\tau_2}(t)} t \int_{\min\{t^{-2/d}a,\d_t\}}^{\delta_t} r^{d-1+\tau_1+\tau_2} \, \dint r \leq C(a,\tau_1,\tau_2)\,,
\end{align*}
and that
\begin{align*}
& \frac{t}{\varrho_{\tau_1}(t)^2 \varrho_{\tau_2}(t)} t \int_{\min\{t^{-2/d}a,\d_t\}}^{\delta_t} r^{d-1+2\tau_1+\tau_2} \, \dint r\\
&\qquad\leq \Bigg(\frac{t}{\varrho_{\tau_1}(t)^4}  t \int_{\min\{t^{-2/d}a,\d_t\}}^{\delta_t} r^{d-1+4\tau_1} \, \dint r \Bigg)^{1/2}\\
&\qquad\qquad\times\Bigg(\frac{t}{\varrho_{\tau_2}(t)^2}  t\int_{\min\{t^{-2/d}a,\d_t\}}^{\delta_t} r^{d-1+2\tau_2} \, \dint r \Bigg)^{1/2}\,.
\end{align*}
Because of
\begin{align*}
t \int_{\min\{t^{-2/d}a,\d_t\}}^{\delta_t} r^{d-1 +4\tau} \, \dint r
\leq \begin{cases} \frac{1}{4\tau+d} t\delta_t^{4\tau+d} &: \tau>-d/4\\
t\ln(t^{2/d}\delta_t)-t\ln(a)  &: \tau=-d/4\\
-\frac{1}{4\tau+d} t^{-1-8\tau/d}a^{4\tau+d} &: \tau<-d/4 \end{cases}
\end{align*}
and the definition of $\varrho_{\tau}(t)$, we obtain that
$$
\frac{t^2}{\varrho_{\tau}(t)^4} \int_{\min\{t^{-2/d}a,\d_t\}}^{\delta_t} r^{d-1 +4\tau} \, \dint r
\leq \begin{cases}
\frac{1}{4\tau+d} \frac{1}{t^2\delta_t^d} &: \tau>-d/4\\
\frac{\ln(t^{2/d}\delta_t)-\ln(a)}{t^2\delta_t^d} &:  \tau=-d/4\\
-\frac{1}{4\tau+d} \frac{a^{4\tau+d}}{(t^2\delta_t^d)^{2+4\tau/d}} &: \tau\in(- \frac{d }{2},- \frac{d }{4})\\
-\frac{1}{4\tau+d} \frac{a^{4\tau+d}}{(\ln(t^{2/d}\delta_t))^2} &: \tau=-d/2\\
-\frac{1}{4\tau+d} \frac{a^{4\tau+d}}{t^{6+8\tau/d}\delta_t^{4\tau+4d}} & : \tau<-d/2\,. \end{cases}
$$
Combining these estimates with \eqref{eqn:BoundMG} and the assumptions on $(\delta_t: t>0)$ leads to
\begin{equation}\label{eqn:ConvergenceMG}
\lim_{t\to\infty}\frac{M_G(a,t)}{\varrho_{\tau_i}(t)^2 \varrho_{\tau_j}(t)^2}=0.
\end{equation}
Now we put $a=1$ and use the convergence of the covariances provided by Theorem \ref{thm:CovariancesUta}, \eqref{eqn:ConvergenceMG} and \eqref{eq:Boundd3} to conclude that
$$
\lim_{t\to \infty} d_3({\bf U}_{t,1},{\bf N}_{\Sigma}) =0
$$
and, thus, ${\bf U}_{t,1} \overset{d}{\longrightarrow} {\bf N}_{\Sigma}$, as $t\to\infty$. On the other hand we have that ${\bf U}_{t,1}-(\widetilde{L}_t^{(\tau_1)}, \ldots, \widetilde{L}_t^{(\tau_m)})$ converges to zero in the $L^1$-sense, as $t\to\infty$, by Lemma~\ref{lem:LtUt}. This completes the proof.
\end{proof}

\begin{proof}[Proof of Theorem~\ref{thm:CLTRateUniv}]
We choose $a=0$ so that $U_{t,a}^{(\tau)}=L_{t}^{(\tau)}$. Note that \eqref{eq:BoundKolmogorov} and the estimates in the previous proof are still true for $a=0$ if $\tau,\tau_1,\tau_2>-d/4$. This yields the desired bound for the Kolmogorov distance.
\end{proof}

\subsection{Stable limit theorems}

After having investigated central limit theorems for the normalized length-powers, we turn now to limit theorems in which a stable random variable takes over the r\^ole of the limiting random variable. To formulate the result, define
$$
\widehat{L}_t^{(\tau)}:=\begin{cases}
t^{2\tau/d} (L_{t}^{(\tau)} - \E L_{t}^{(\tau)}) &: \tau\in(-d,- \frac{d }{2})\\
t^{-2} (L_t^{(-d)}-E_t ) &: \tau=-d\\
t^{2\tau/d} L_t^{(\tau)} &: \tau<-d
\end{cases}
$$
with
$$
E_t:= \E \frac{1}{2}\sum\limits_{(x,y)\in\eta_{t,\neq}^2} \I\Big(t^{-2/d} \Big(\frac{2}{\kappa_d V(W)}\Big)^{1/d}\leq \|x-y\| \leq \delta_t\Big)\|x-y\|^{-d}.
$$
For a unit-intensity Poisson point process $\zeta$ on $\R_+$, we define the random variables
$$
Z^{(\tau)} :=\begin{cases}
\lim\limits_{n\to\infty}\Big(\frac{2}{\kappa_d V(W)} \Big)^{\tau/d} \Big( \sum\limits_{x\in \zeta \cap [0,n]} x^{\tau/d} - \frac{1}{1+\tau/d} n^{1+\tau/d} \Big) &: \tau\in(-d,- \frac{d }{2})\\
\lim\limits_{n\to\infty}\frac{\kappa_d V(W)}{2} \Big(\sum\limits_{x\in\zeta\cap [0,n]} x^{-1} - \ln n \Big) &: \tau=-d\\
\Big(\frac{2}{\kappa_d V(W)}\Big)^{\tau/d} \sum\limits_{x\in\zeta} x^{\tau/d} &: \tau<-d\,,
\end{cases}
$$
where $\lim\limits_{n\to\infty}$ indicates the a.s.\ limit for $n\in\N$. We remark that the so-defined random variables $Z^{(\tau)}$ are $\alpha$-stable with $\alpha=d/|\tau|\in(0,2)$.

To complement the central limit theorems from the previous section, we now consider length-power functionals $L_t^{(\tau)}$ with powers $\tau \leq -d$ and also $\tau\in(-d,-d/2)$ in case that $t^{3+ \frac{4\tau }{d}}\d_t^{2(d+\tau)}\to 0$. Recall that the latter expression is required to converge to infinity to have a central limit theorem.

\begin{theorem}\label{thm:StableMulti}
Let $-d/2>\tau_1>\tau_2>\ldots>\tau_m$ for $m\in\N$ and assume that $t^2\d_t^d\to\infty$ and
$t^{3+ \frac{4\tau_1 }{d}}\d_t^{2(d+\tau_1)}\to 0$, as $t\to\infty$. Then,
$$
\big(\widehat{L}_t^{(\tau_1)},\ldots,\widehat{L}_t^{(\tau_m)}\big)\overset{d}{\longrightarrow}\big(Z^{(\tau_1)},\ldots,Z^{(\tau_m)}\big)\,,\qquad\text{as}\qquad t\to\infty\,.
$$
\end{theorem}

For an univariate version of Theorem \ref{thm:StableMulti} in the case that $\tau<-d$ we refer to Corollary 7.10 in \cite{DST}. The proof of Theorem~\ref{thm:StableMulti} is prepared by the following lemma.

\begin{lemma}\label{lem:L2approximation}
For $\tau \geq -d$, $a>0$ and $t,\delta_t>0$ such that $t^{-2/d}a\leq \delta_t$ one has that
\begin{align*}
& \V \frac{1}{2} \sum_{(x,y)\in\eta^2_{t,\neq}} \I(t^{-2/d}a \leq \|x-y\|\leq \delta_t) \|x-y\|^\tau\\
& \leq t^3 V(W) d^2 \kappa_d^2 \bigg( \int_{t^{-2/d}a}^{\delta_t} r^{\tau+d-1} \, \dint r \bigg)^2 + \frac{t^2}{2} V(W) d\kappa_d \int_{t^{-2/d}a}^{\delta_t} r^{2\tau+d-1} \, \dint r\,.
\end{align*}
\end{lemma}
\begin{proof}
By the same arguments as in the proof of Theorem~\ref{thm:AsymptoticCovariance}, we obtain that
\begin{align*}
& \V \frac{1}{2} \sum_{(x,y)\in\eta^2_{t,\neq}} \I(t^{-2/d}a \leq \|x-y\|\leq \delta_t) \|x-y\|^\tau\\
& = t^3 \int\limits_W \Bigg( \int\limits_W \I(t^{-2/d}a\leq \|x-y\|\leq \delta_t) \|x-y\|^\tau \, \dint y \Bigg)^2 \, \dint x \\
&\qquad\qquad + \frac{t^2}{2} \int\limits_{W^2} \I(t^{-2/d}a \leq \|x-y\|\leq \delta_t) \|x-y\|^{2\tau} \, \dint (x,y) \allowdisplaybreaks\\
& \leq t^3 V(W) d^2 \kappa_d^2 \bigg( \int_{t^{-2/d}a}^{\delta_t} r^{\tau+d-1} \, \dint r \bigg)^2 + \frac{t^2}{2} V(W) d\kappa_d \int_{t^{-2/d}a}^{\delta_t} r^{2\tau+d-1} \, \dint r\,.
\end{align*}
This yields the result.
\end{proof}

\begin{proof}[Proof of Theorem~\ref{thm:StableMulti}]
Throughout this proof we use the abbreviation $c_W:=2/(\kappa_dV(W))$. Let $M\in\{1,\ldots,m\}$ be such that $-d/2>\tau_1>\ldots>\tau_{M-1}>\tau_{M}=-d>\tau_{M+1}>\ldots>\tau_m$, whenever such an index exists (if not, some of the cases considered below do not occur and can therefore be omitted). Instead of the functionals $L_t^{(\tau_1)},\ldots,L_{t}^{(\tau_m)}$, we first deal with their truncated versions $L_{t,a}^{(\tau_1)},\ldots,L_{t,a}^{(\tau_m)}$ with $a>0$; see \eqref{eqn:DefinitionLta} for the definition. For $i\in\{1,\ldots,M-1\}$ we notice that
\begin{align*}
\lim_{t\to\infty} t^{2\tau_i/d}\,\E L_{t,a}^{(\tau_i)}
& = \lim_{t\to\infty} \frac{t^{2+2\tau_i/d}}{2} \int\limits_{W^2} \I(\|x-y\|\leq t^{-2/d}a) \|x-y\|^{\tau_i} \, \dint(x,y)\\
& = \frac{d\kappa_d V(W)}{2} \int_0^{a} r^{\tau_i+d-1} \, \dint r = \frac{1}{c_W(1+\tau_i/d)} a^{\tau_i+d}\,.
\end{align*}
Moreover, we define
$$
E_{t,a}:= \E \frac{1}{2}\sum\limits_{(x,y)\in\eta_{t,\neq}^2} \I(t^{-2/d} c_W^{1/d}\leq \|x-y\| \leq \min\{t^{-2/d}a,\delta_t\})\|x-y\|^{-d}
$$
and see that
\begin{align*}
 \lim_{t\to\infty} t^{-2} E_{t,a} 
& = \lim_{t\to\infty } \frac{1}{2} \int\limits_{W^2} \I(t^{-2/d} c_W^{1/d}\leq \|x-y\| \leq t^{-2/d}a)\|x-y\|^{-d} \, \dint(x,y)\\
& = \frac{d \kappa_d}{2} V(W) \int_{c_W^{1/d}}^a r^{-d+d-1} \, \dint r  = \frac{d}{c_W} (\ln a - \ln c_W^{1/d} )= \frac{1}{c_W} \ln \frac{a^d}{c_W}\,.
\end{align*}
Now, an application of Theorem~\ref{thm:PPPConv} shows that the random vector
\begin{equation}\label{eqn:vectorLta}
\bigg( \Big(t^{2\tau_i/d} (L_{t,a}^{(\tau_i)} - \E L_{t,a}^{(\tau_i)})\Big)_{i=1}^{M-1}, t^{-2}(L_{t,a}^{(-d)}-E_{t,a}), \Big(t^{2\tau_{i}/d}L_{t,a}^{(\tau_{i})}\Big)_{i=M+1}^m\bigg)
\end{equation}
converges in distribution, as $t\to\infty$, to the random vector
\begin{equation*}
\begin{split}
{\bf Z}_a:=&\bigg( \Big( c_W^{\tau_i/d}\sum_{x\in \zeta \cap [0,a^d/c_W]} x^{\tau_i/d} - \frac{a^{d+\tau_i}}{c_W(1+\tau_i/d)} \Big)_{i=1}^{M-1},\\
& c_W^{-1} \sum_{x\in\zeta\cap [0,a^d/c_W]} x^{-1} - c_W^{-1} \ln \frac{a^d}{c_W}, \Big(c_W^{\tau_{i}/d}\sum_{x\in\zeta\cap[0,a^d/c_W]}x^{\tau_{i}/d}\Big)_{i=M+1}^m\bigg)\,.
\end{split}
\end{equation*}
Next, we notice that Lemma~\ref{lem:L2approximation} implies that
$$
\V[ t^{2\tau_i/d}(L^{(\tau_i)}_t-L_{t,a}^{(\tau_i)}) ]\leq \frac{d^2\kappa_d^2 V(W)}{(\tau_i+d)^2} t^{3+4\tau_i/d} \delta_t^{2\tau_i+2d}+\frac{d\kappa_d V(W)}{2(-2\tau_i-d)}a^{2\tau_i+d}
$$
for all $i\in\{1,\ldots,M-1\}$, which tends to zero, as $t\to\infty$ and $a\to\infty$, by our assumptions on $\tau_1,\ldots,\tau_{M-1}$ and $(\d_t:t>0)$. It also yields that
\begin{align*}
& \V[ t^{-2} (L_t^{(-d)} -E_t - L_{t,a}^{(-d)}+E_{t,a}) ] \\
& \leq d^2\kappa_d^2 V(W) t^{-1} (\ln \delta_t - \ln(t^{-2/d} a))^2  + \frac{\kappa_d V(W)}{2} a^{-d}\,,
\end{align*}
which tends to zero as well, as $t\to\infty$ and $a\to\infty$ such that $\ln(a)/t\to 0$. Finally, a similar computation as in the proof of Theorem~\ref{thm:Expectation} shows that for $i\in\{M+1,\ldots,m\}$ and $t^{-2/d}a\leq \delta_t$,
\begin{align*}
\E|t^{2\tau_i/d}(L_t^{(\tau_i)}-L_{t,a}^{(\tau_i)})| & \leq \frac{d\kappa_dV(W) t^{2+2\tau_i/d}}{2} \int_{t^{-2/d}a}^{\delta_t} r^{\tau_i+d-1} \, \dint r\\
& \leq \frac{d\kappa_dV(W) t^{2+2\tau_i/d}}{2(\tau_i+d)} (t^{-2/d}a)^{\tau_i+d} = \frac{d\kappa_dV(W) }{ 2(\tau_i+d)}\,a^{d+\tau_i}\,.
\end{align*}
Here, the right-hand side also tends to zero, as $a\to\infty$, by our assumptions on $\tau_{M+1},\ldots,\tau_m$. Summarizing, we have shown that, as $t\to\infty$ and $a\to\infty$ such that $\ln(a)/t\to 0$, the random vector in \eqref{eqn:vectorLta} converges to $(\widehat{L}_t^{(\tau_1)},\hdots,\widehat{L}_t^{(\tau_m)})$ in distribution. Since the random vector ${\bf Z}_a$  converges in distribution to the random vector $(Z^{(\tau_1)},\ldots,Z^{(\tau_m)})$, as $a\to\infty$, this completes the proof of the theorem.
\end{proof}

\begin{remark}\rm
In view of Theorem~\ref{thm:CLTMulti} and Theorem~\ref{thm:StableMulti} it remains to investigate $L_t^{(\tau)}$ with $\tau\in(-d,-d/2)$ in the case that $t\to\infty$, $t^2\delta_t^d\to\infty$ and $t^{3+ \frac{4\tau }{d}}\delta_t^{2(d+\tau)}\to c\in(0,\infty)$. In this situation we can decompose $L_t^{(\tau)}$ as
\begin{equation}\label{eq:Zwischenfall}
\begin{split}
L_t^{(\tau)} = & \frac{1 }{2}\sum_{(x,y)\in\eta_{t,\neq}^2}\I(\|x-y\|<t^{-2/d}a)\,\|x-y\|^\tau\\
 & + \frac{1 }{2}\sum_{(x,y)\in\eta_{t,\neq}^2}\I(t^{-2/d}a\leq\|x-y\|\leq\delta_t)\,\|x-y\|^\tau\,,
\end{split}
\end{equation}
where $a>0$. For $a\to\infty$ and $t\to\infty$ in the right way one can argue as in the proofs of Theorem \ref{thm:CLTMulti} and Theorem \ref{thm:StableMulti} that under a suitable normalization the second term in \eqref{eq:Zwischenfall} converges to a Gaussian random variable, while the first term in \eqref{eq:Zwischenfall} converges to a stable random variable. Since we are not able to describe the joint distribution of the two limiting random variables, we do not further investigate this exceptional case.
\end{remark}

\subsection{Compound Poisson limit theorems}

In the previous subsections we derived multivariate central and stable limit theorems for length-power functionals under the hypothesis that $t^2\delta_t^d\to\infty$, as $t\to\infty$.
We shall in this subsection discuss the remaining case in which $t^2 \delta_t^d\to c\in[0,\infty)$, as $t\to\infty$. The following result is a reformulation of Theorem \ref{thm:PPPConv} with $a=c^{1/d}$.

\begin{theorem}\label{thm:CompoundPoisson}
Assume that $t^2 \, \delta_t^d \to c \in [0,\infty)$, as $t\to\infty$, and let $(X_j)_{j\in\N}$ be an i.i.d.\ family of uniformly distributed random variables in $[0,c]$ and let $N$ be a Poisson random variable with mean $\kappa_d V(W)c/2$ independent of $(X_j)_{j\in\N}$. Then, as $t\to\infty$,
$$
\Big( t^{2\tau_1/d} L_t^{(\tau_1)},\ldots,t^{2\tau_m/d} L_t^{(\tau_m)} \Big) \overset{d}{\longrightarrow} \Big( \sum_{j=1}^N X_j^{\tau_1/d},\ldots,\sum_{j=1}^N X_j^{\tau_m/d}\Big)
$$
for $\tau_1,\hdots,\tau_m\in\R$, $m\in\N$.
\end{theorem}

Note that the components of the limiting random vector in Theorem \ref{thm:CompoundPoisson} are com\-pound Poisson distributed. For $m=1$ and $\tau_1=0$, this yields that $L_t^{(0)}$ converges in distribution to a Poisson random variable with mean $\kappa_d V(W)c/2$. This special case is discussed in Theorem 4.12 in \cite{LachiezeReyPeccatiI} and in Theorem 5.1 in \cite{Peccati2011} a rate of convergence is derived in this univariate limit theorem. A one-dimensional quantitative version of Theorem \ref{thm:CompoundPoisson} can be found as Corollary 7.6 in \cite{DST}, where the rate of convergence is measured by the total variation distance.

\section{Concentration inequalities}\label{sec:LDI}

This section concerns concentration inequalities for the functionals $L_t^{(\tau)}$. This was a widely open field for a long time due to the lack of general deviation inequalities, or even efficient methods to prove deviation inequalities for special functionals.
In the last years a first attempt was made by Eichelsbacher, Rai\v{c} and Schreiber \cite{EichelsbacherRaicSchreiber2013} proving deviation inequalities for a class of local Poisson functionals. Their result can be applied to the random variable $L_t^{(\tau)}$ for special choices of $\tau$ and $(\delta_t:t>0)$.

Being more precise, in the paper by Eichelsbacher, Rai\v{c} and Schreiber \cite{EichelsbacherRaicSchreiber2013} deviation inequalities for {\it stabilizing} Poisson functionals were derived. But due to the dependence on $t$ and $\delta_t$ in the present paper, this result can be only applied in the thermodynamic regime, where $\delta_t=\tilde{\delta} t^{-1/d}$ with some fixed $\tilde{\delta}>0$. In this case, the following concentration inequality can be obtained, whose proof is postponed to the end of Subsection~\ref{subsec:Poisson}.

\begin{proposition}\label{prop:ERS}
Suppose that $(\delta_t:t>0)$ is such that $\delta_t=\tilde{\delta}t^{-1/d}$ for some fixed $\tilde{\delta}>0$. Then for any $\tau\geq 0$ there is a constant $c>0$ depending on $\tau$, $W$ and $\tilde{\delta}$ such that
 \begin{equation*}
\P( | L^{(\tau)}_t  - \E L^{(\tau)}_t |\geq u)
\leq
{\rm exp}\big(-c\,\min\big\{ t^{\frac{2\tau-d}d}u^2,\,t^{\frac{\tau}{3d}} u^{\frac 13},\,t^{\frac{3\tau-d}{4d}}u^{\frac 34}\big\} \big)
\end{equation*}
for $u>0$ and $t\geq 1$.
\end{proposition}

The question to prove deviation inequalities in the general case was still open. Here we introduce a new method, linking $L_t^{(\tau)}$ to the {\it convex distance} of Poisson point processes and then applying a recent deviation inequality for the convex distance. This yields -- see Theorem~\ref{thm:LDIPoisson} --  for the first time concentration inequalities for general $\delta_t$ and $ \tau \geq 0$.

Continuing our line of research, a very recent progress was the work of Bachmann and Peccati \cite{BachmannPeccati}, proving large deviation inequalities for certain general Poisson functionals. These yield exponential upper bounds for the upper tail probability $\P( L^{(\tau)}_t  - \E L^{(\tau)}_t \geq u)$ in the case where $\tau\in[0,1]$. More precisely, Corollary 7.4 in \cite{BachmannPeccati} says that
\begin{equation*}
\P( L^{(\tau)}_t  - \E L^{(\tau)}_t \geq u) \leq {\rm exp}(-I_1(u))\,,
\end{equation*}
where $I_1(u)\approx cu^{\frac 12} (\ln u)^{\frac 12} $ with a constant $c=c(t,\delta_t,\tau)\in(0,\infty)$ that does not depend on $u$ but on the intensity parameter $t$, the distance threshold $\delta_t$ and the exponent $\tau$. Here, recall, we write $f(u)\approx g(u)$ if $f(u)/g(u)\to 1$, as $u\to\infty$. The precise form of $I_1(u)$ is rather complicated and to deduce the dependence on $t$ and $\d_t$ seems to be demanding.
In a follow--up paper to our investigations presented here and the paper by  Bachmann and Peccati,  Bachmann and Reitzner \cite{BachmannReitzner} are generalizing both approaches to subgraph counts of the Gilbert graph.

\medskip
In the last subsection we modify our approach to settle also the corresponding question for a binomial process.

\iffalse
These results should be compared with the result of Proposition~\ref{prop:ERS}, which provides the weaker bound $\P( | L^{(\tau)}_t  - \E L^{(\tau)}_t |\geq u)\leq{\rm exp}(-I_2(u))$ with $I_2(u)\approx c_2u^{1/3}\,t^{\tau/(3d)}$ and a constant $c_2\in(0,\infty)$ that is independent of $u$. A comparison with our results is postponed until Remark~\ref{rem:Comparison}.
\fi

\subsection{The Poisson point process and the convex distance}\label{subsec:Poisson}

In this subsection we consider the case of an underlying homogeneous Poisson point process of intensity $t>0$ within a compact convex observation window $W$ having interior points and prove the following concentration inequality for $L^{(\tau)}_t$.

\begin{theorem}\label{thm:LDIPoisson}
Let $\tau\geq 0$, let $m_t$ be the median of $L_t^{(\tau)}$ and define
\begin{eqnarray*}
x^*(u,t)
&=&
\inf_{s>0}\frac{\ln(t V(W))+t\kappa_d \delta_t^d(e^s-1)}{2 t\kappa_d \delta_t^d s}
\\ &&+
\sqrt{\frac{u^2}{8 t^2\kappa_d^2 \delta_t^{2d+\tau} (u+m_t)s}
+\bigg( \frac{\ln(tV(W)) +t\kappa_d \delta_t^d(e^s-1)}{2 t\kappa_d \delta_t^d s} \bigg)^2}\,.
\end{eqnarray*}
Then, for $u>0$,
\begin{equation}\label{eq:LDIPoisson}
\P( | L^{(\tau)}_t -m_t |\geq u) \leq 8 \exp{-\frac{u^2}{8 t\kappa_d\delta_t^{d+\tau}x^*(u,t)(u+m_t)}}
\end{equation}
and, in particular,
\begin{equation}\label{eq:LDIPoissons1}
\P( | L^{(\tau)}_t -m_t |\geq u)
\leq
8 \exp{-\min\bigg\{\frac {u^2 }{ C_{t,W} (u+m_t)},\frac {u }{ D_t\sqrt{(u+m_t)}}\bigg\}}
\end{equation}
with
$ C_{t,W}= 16 (\ln(tV(W))\d_t^{\tau} +2t\kappa_d \delta_t^{d+\tau})$ and
$D_t =  4 \sqrt{2 \d_t^\tau}$.
\end{theorem}

\begin{remark}\label{rem:Comparison}\rm
In the thermodynamic regime we can compare our bound with that of Eichelsbacher, Rai\v{c} and Schreiber stated in Proposition~\ref{prop:ERS}. The second inequality in Theorem~\ref{thm:LDIPoisson} and $m_t\leq 2\E L_t^{(\tau)}$ deliver up to a constant factor the exponent
$$
-\min\Big\{ \frac{u^2\,t^{\frac{2\tau}d} }{t\,\ln t}\,,  \frac{u\,t^{\frac{\tau}d} }{ \sqrt{t}}\,,\sqrt{u}\,t^{\frac{\tau}{2d}}\Big\}\,,
$$
whereas Proposition~\ref{prop:ERS} yields
$$
-\min\Big\{ \frac{u^2\,t^{\frac{2\tau}d} }{t}\,,  \frac{u^{\frac 34}\,t^{\frac{3\tau}{4d}} }{t^{\frac{1}4}}\,,u^{\frac{1}3}\,t^{\frac{\tau}{3d}}\Big\}\,.
$$
This means that for $u$ such that $u^2$ dominates the minimum, the re-scaling with respect to the intensity parameter $t$ in the result of Eichelsbacher, Rai\v{c} and Schreiber is better than ours. On the other hand, our worst $u$-exponent is $1/2$, whereas it is $1/3$ in Proposition~\ref{prop:ERS}. Moreover, one should notice that Proposition~\ref{prop:ERS} deals with an inequality for $L_t^{(\tau)}-\E L_t^{(\tau)}$, whereas in our result $\E L_t^{(\tau)}$ is replaced by the median $m_t$.
\iffalse
In comparison with the result of Bachmann and Peccati from \eqref{eq:BachmannPeccati} we stress that our inequality works for all $\tau\geq 0$ and that Theorem~\ref{thm:LDIPoisson} delivers the asymptotic exponential exponent $u^{1/2}t^{\tau/(2d)}$, which is the same as that in \eqref{eq:BachmannPeccati} up to a logarithmic factor. In fact, up to a logarithmic correction factor, Corollary 6.2 in \cite{BachmannPeccati} shows that this asymptotic exponential exponent is optimal.
Yet the lower tail bound \eqref{eq:BachmannPeccatilower} from Bachmann and Peccati is much better than ours.
\fi
\end{remark}

\bigskip
For the proof of Theorem \ref{thm:LDIPoisson} a local version of $L_t^{(\tau)}$ plays an essential r\^ole. Namely, for a finite counting measure $\nu$ and $x \in \nu$ let us define
$$ L_t^{(\tau)} (x; \nu) :=  \sum_{y \in \nu} \I(\| x-y\| \leq \d_t) \| x-y \| ^\tau , $$
and thus we may write
$ L^{(\tau)}_t(\nu) = \frac 12\sum_{x\in\nu} L_t^{(\tau)} (x; \nu) $, and hence  $L_t^{(\tau)}(\eta_t) = L_t^{(\tau)}$. For two finite counting measures $\nu$ and $ \zeta$ we define the (set-)difference $\nu \backslash \zeta$ by
\begin{equation*}
\nu \backslash \zeta := \sum_{x \in W} \max\{\nu(x)-\zeta(x),0\} \, \varepsilon_x,
\end{equation*}
which is again a finite counting measure. Assume now that besides of $\eta_t$ a second point set $\zeta \in \bN(W)$ is chosen, which might have a non-trivial intersection with $\eta_t$. Each edge between points in $\eta_t$ either belongs to an edge between points in $\zeta$ if both endpoints are contained in $\eta_t \cap \zeta$, or is  counted at least once in some $L_t^{(\tau)}(x;\eta_t)$ for an endpoint $x \in {\eta_t} \backslash \zeta$. We thus find
\begin{equation}\label{eq:UGvonWeiterOben}
L^{(\tau)}_t(\eta_t)\leq L_t^{(\tau)} (\zeta)  + \sum_{x \in {\eta_t}\backslash \zeta} L_t^{(\tau)}(x; \eta_t)\,.
\end{equation}
To derive our deviation inequality we use an analogue of Talagrand's convex distance for Poisson point processes, which has been introduced in \cite{Reitzner13}.
For $\nu \in \bN(W) $ and $ A \subset \bN(W)$ it is given by
$$
d_T^\pi ({\nu},A)  : = \max_{ \| u \|_{2, \nu} \leq 1 }\  \min_{ \zeta \in A}
\sum_{x\in\nu \backslash \zeta} u(x)\,,
$$
where $u: W \to \R_+$ is a non-negative measurable function and $\| u \|_{2 ,\nu}^2:= \sum_{x\in\nu} u(x)^2$. We now assume that $L_t^{(\tau)}(\eta_t)\neq 0$.
For $x \in W$, let us put
$$
u(x) := \frac{1}{\| L_t^{(\tau)} (\,\cdot\,; \eta_t) \|_{2, \eta_t} } L_t^{(\tau)} (x; \eta_t)\,,
$$
which gives $\| u \|_{2, \eta_t}^2=1$.
Using \eqref{eq:UGvonWeiterOben} we rewrite $L^{(\tau)}_t$ in terms of the convex distance as follows:
\begin{equation}\label{eq:ZwischenschrittEingefuegt}
\begin{split}
d_T^\pi ({\eta_t},A)
&=
\max_{ \| u \|_{2, \eta_t} \leq 1 }\  \min_{ \zeta \in A}
\sum_{x\in \eta_t \backslash \zeta} u(x)\\
& \geq
\min_{ \zeta \in A}
 \frac{1}{\| L_t^{(\tau)} (\,\cdot\,; \eta_t) \|_{2, \eta_t} } \sum_{x\in \eta_t \backslash \zeta}L_t^{(\tau)} (x; \eta_t)
\\ & \geq \min_{ \zeta \in A}
\frac{1}{\| L_t^{(\tau)} (\,\cdot\,; \eta_t) \|_{2, \eta_t} }
\Big( L_t^{(\tau)}(\eta_t) - L_t^{(\tau)}  (\zeta) \Big)\,.
\end{split}
\end{equation}
We now assume that $L_t^{(\tau)} (x; \eta_t) \leq B $ for all $x \in \eta_t$ for some $B>0$, in which case
$$ \| L_t^{(\tau)} (\,\cdot\,; \eta_t) \|_{2, \eta_t}^2  =
\sum_{x\in \eta_t} L_t^{(\tau)} (x; \eta_t)^2 \leq
B  \sum_{x\in\eta_t} L_t^{(\tau)} (x; \eta_t)
= 2B L_t^{(\tau)}(\eta_t)\,.$$
In view of \eqref{eq:ZwischenschrittEingefuegt} this gives
\begin{equation}\label{eq:dTpi}
d_T^\pi ({\eta_t},A)\geq \frac{1}{\sqrt {2 B } }
\min_{ \zeta \in A}  \frac{L^{(\tau)}_t(\eta_t) - L_t^{(\tau)}  (\zeta)}{ \sqrt{L_t^{(\tau)} ({\eta_t})}}
\end{equation}
if $L_t^{(\tau)}(\eta_t)\neq 0$, $L_t^{(\tau)}(x;\eta_t)\leq B$ for all $x\in\eta_t$ and $L_t^{(\tau)}(\eta_t)\geq L_t^{(\tau)}(\zeta)$ for all $\zeta\in A$.

Inequality (\ref{eq:dTpi}) links $d_T^\pi$ to $L^{(\tau)}_t$. We are now in the position to recall the main result from \cite{Reitzner13}, in which Talagrand's large deviation inequality for binomial point processes was extended to finite Poisson point processes. Namely, for $A \subset {\bf N}(W)$ and $s\geq 0$ we have that
\begin{equation}\label{eq:TalagrandPoiss}
 \P (A) \P  \left( d_T^\pi({\eta_t}, A) \geq s  \right) \leq  \exp{- \frac {s^2} 4} \, .
\end{equation}

Let us make the relation between $d_T^\pi $ and $L^{(\tau)}_t(\eta_t)$  more explicit.
First, define
$ A = \{ \nu \in \bN(W):\ L^{(\tau)}_t (\nu) \leq m_t \} $,
where $m_t$ is the median of $L^{(\tau)}_t(\eta_t)$, and let $u>0$. By definition of the median, we have $ \P(A) \geq \frac 12$.
Since the function $s\mapsto {s}/{\sqrt {s+ m_t}}$ is increasing, $ L^{(\tau)}_t(\eta_t) \geq u+ m_t$ and
 \eqref{eq:dTpi} imply that
$d_T (\eta_t,A)   \geq
\frac{1}{\sqrt{2B}} \ \frac{u}{\sqrt{u+ m_t}}  $.
Together with  \eqref{eq:TalagrandPoiss}, this yields
\begin{align}
& \P( L^{(\tau)}_t (\eta_t) \geq u+ m_t) \nonumber\\
& \leq   \nonumber
\P( L^{(\tau)}_t (\eta_t) \geq u+ m_t,\ \forall x \in \eta_t:\  L^{(\tau)}_t(x;\eta_t) \leq B  )\\
& \quad +  \P( \exists x \in \eta_t:\ L_t^{(\tau)}(x; \eta_t) >B) \nonumber\\
 &\leq   \nonumber
\P  \Big( d_T(\eta_t,A)  \geq \frac 1{\sqrt {2B }} \frac {u }{\sqrt{ u+ m_t}} \Big)
+  \P( \exists x \in \eta_t:\ L_t^{(\tau)}(x; \eta_t) >B)
\\ & \leq  \label{eq:LDINt1}
2 \exp{- \frac{u^2} {8B (u+m_t) } }
+  \P( \exists x \in \eta_t:\ L_t^{(\tau)}(x; \eta_t) >B).
\end{align}
In the following we assume that $m_t\neq0$. Similarly as above, for $u \in (0,m_t]$, putting $A=\{\nu:\  L^{(\tau)}_t(\nu) \leq m_t-u\}$, the monotonicity of $(s-a)/\sqrt s$ for $a\geq 0$ together with the assumption that $L^{(\tau)}_t (\eta_t) \geq m_t$ imply
$$
\frac{L^{(\tau)}_t(\eta_t) -(m_t-u)}{\sqrt{L^{(\tau)}_t(\eta_t)}} \geq  \frac{m_t -(m_t-u)}{\sqrt{ m_t}}\,.
$$
In this case we have
$$ d_T( \eta_t, A) \geq \frac{1}{\sqrt{2B }}\min_{\zeta \in A}  \frac {L^{(\tau)}_t(\eta_t)- L^{(\tau)}_t(\zeta)}{\sqrt{L^{(\tau)}_t(\eta_t)}}
\geq \frac{1}{\sqrt{2B }}\  \frac u{\sqrt{ m_t}}\,,
$$
which, again in view of \eqref{eq:TalagrandPoiss}, yields
\begin{align*}
\frac 12 \leq \P( L^{(\tau)}_t(\eta_t) \geq m_t)
& \leq \P( L^{(\tau)}_t (\eta_t) \geq m_t,\ \forall x\in\eta_t:\  L^{(\tau)}_t(x;\eta_t) \leq B  ) \\ 
& \hskip3cm +  \P( \exists x \in \eta_t:\ L_t^{(\tau)}(x; \eta_t) >B)\\ 
&\leq \P \Big( d_T( \eta_t, A)  \geq \frac 1{\sqrt{2B }} \ \frac u{\sqrt{ m_t}} \Big)\\
& \hskip3cm + \P( \exists x \in \eta_t:\ L_t^{(\tau)}(x; \eta_t) >B)\\
&\leq \frac 1{\P(L^{(\tau)}_t(\eta_t) \leq m_t-u)} \exp{- \frac {u^2}{8B m_t}}\\
& \hskip3cm + \P( \exists x \in \eta_t:\ L_t^{(\tau)}(x; \eta_t) >B),
\end{align*}
and we deduce that
$$
\P(L^{(\tau)}_t(\eta_t) \leq m_t-u)\leq2 \exp{\!\!- \frac {u^2}{8B m_t}}+  2 \P( \exists x \in \eta_t:\ L_t^{(\tau)}(x; \eta_t) >B)\,.
$$
For $m_t=0$ and $u>0$ the probability on the left-hand side is zero. Combining the previous inequality with (\ref{eq:LDINt1}) implies the following concentration inequality:

\begin{proposition}\label{th:LDIallgPP}
For $\tau\in\R$ and $u > 0 $ we have
\begin{eqnarray*}
\P( | L^{(\tau)}_t -m_t |\geq u)
&\leq&
\inf_{B>0} \Big\{
4 \exp{- \frac {u^2 }{ 8 B (u+m_t) }}
\\ \nonumber && \hskip2.8cm
+ 3\,  \P( \exists x \in \eta_t:\ L_t^{(\tau)}(x; \eta_t) >B) \Big\}
\end{eqnarray*}
where $m_t$ is the median of $L^{(\tau)}_t$.
\end{proposition}

In the next step we provide a bound for $\P( \exists x \in \eta_t:\ L_t^{(\tau)}(x; \eta_t) >B)$.
\begin{lemma}\label{le:P-ex-x}
For $\tau\geq 0$ and $B>0$ we have
$$ \P( \exists x \in \eta_t:\ L_t^{(\tau)}(x; \eta_t) >B)
\leq
t V(W) \inf_{s\geq 0} \exp{E(e^s-1)-s\d_t^{- \tau}B}
$$
with $E=t \kappa_d \d_t^d  $.
\end{lemma}
\begin{proof}
Observe that for some fixed $x\in W$, $L_t^{(0)}(x; \eta_t\cup \{x\} ) = \eta_t(B^d(x, \d_t))$ is a Poisson distributed random variable with mean
\begin{equation*}
 \E \eta_t(B^d(x, \d_t))  = t  V(B^d(x, \d_t)\cap W) \leq t \kappa_d \d_t^d  =: E\,.
\end{equation*}
Mecke's formula \eqref{eq:Mecke} gives, for $B>0$,
\begin{eqnarray*}
\P( \exists x \in \eta_t:\ L_t^{(\tau)}(x; \eta_t) >B)
&\leq&
\E \sum_{x\in\eta_t} \I(L_t^{(\tau)}(x; \eta_t) >B)
\\ &\leq&
 t \int\limits_W \P( \eta_t(B^d(x, \d_t)) > \d_t^{- \tau} B) \, \dint x
\\ & \leq  &
t V(W) \inf_{s\geq 0} \exp{E(e^s-1)-s\d_t^{- \tau}B}\,,
\end{eqnarray*}
where in the last line we used
the Chernoff bound for the Poisson distribution: if $X$ is a Poisson random variable with mean $a>0$ and $u> 0$, we have that
$$
\P(X\geq u)\leq \inf_{s\geq 0}\E \exp{sX-su}=\inf_{s\geq 0}\exp{a(e^s-1)-su}.
$$
This completes the proof.
\end{proof}

\begin{proof}[Proof of Theorem~\ref{thm:LDIPoisson}]
Lemma~\ref{le:P-ex-x} and Proposition~\ref{th:LDIallgPP} with the substitution $B=t\kappa_d\delta_t^{d+\tau}x$ give
\begin{eqnarray*}\label{eq:BoundPoisson}
\P( | L^{(\tau)}_t -m_t |\geq u)
&\leq&
\inf_{s\geq 0, x>0}
4 \exp{- \frac {u^2 }{ 8 t \kappa_d \d_t^{d+\tau} (u+m_t) x}}
\\ && \nonumber \hskip2.2cm
+  4 t V(W) \exp{t \kappa_d \delta_t^d (e^s-1-sx)}\,.
\end{eqnarray*}
Note that the right-hand side is minimized up to a constant factor if both summands are of the same order, i.e., if
$$
- \frac {u^2 }{ 8 t \kappa_d \d_t^{d+\tau} (u+m_t)x} = \ln(tV(W))+t\kappa_d\delta_t^d (e^s-1-sx)\,.
$$
This is equivalent to say that
\begin{eqnarray*}
 x
 &=&
 \frac{\ln(t V(W))/(t\kappa_d\delta_t^d)+e^s-1}{2s}
 \\ &&  \nonumber
 +\sqrt{\frac{u^2}{8 \kappa_d^2 t^2 \delta_t^{2d+\tau} (u+m_t)s}+\bigg( \frac{\ln(tV(W))/(\kappa_dt\delta_t^d)+e^s-1}{2s} \bigg)^2}\,,
\end{eqnarray*}
since the negative solution of the quadratic equation above does not satisfy $x>0$, as required. Now, choosing $s>0$ in such a way that $x$ is minimized leads to \eqref{eq:LDIPoisson}. For $s=1$ we obtain
$$
x\leq 2\max\bigg\{ \frac{\ln(t V(W)) }{t\kappa_d\delta_t^d}+2,\sqrt{\frac{u^2}{8 \kappa_d^2 t^2 \delta_t^{2d+\tau} (u+m_t)}} \bigg\}\,,
$$
which yields \eqref{eq:LDIPoissons1}.
\end{proof}

\begin{proof}[Proof of Proposition~\ref{prop:ERS}]
For $x\in\R^d$ and a finite counting measure $\nu$ we define the score function
$$
\xi(x,\nu)=\frac{1}{2}\sum_{y\in\nu}\I(\|x-y\|\leq \tilde{\delta}) \|x-y\|^\tau
$$
for $\tilde{\delta}>0$ and the rescaled score function $\xi_t(x,\nu)$ as $\xi_t(x,\nu):=\xi(t^{1/d}x, t^{1/d}\nu)$ for $t\geq 1$ (to simplify comparison with \cite{EichelsbacherRaicSchreiber2013} we use the same notation as in that paper, which should not be confused with the notation used at the beginning of this subsection). Hence, we can rewrite $L_t^{(\tau)}$ as
$$
L_t^{(\tau)} = t^{-\frac \tau d}\sum_{x\in\eta_t} \xi_t(x,\eta_t)\,.
$$
Note that $\xi_t$ has $t^{-1/d}\tilde{\delta}$ as its so-called radius of stabilization (see \cite{EichelsbacherRaicSchreiber2013}) and that the number of points of $\eta_t$ which affect the value $\xi_t(x,\eta_t)$ is Poisson distributed for all $t\geq 1$. Consequently, for
$$
t^{\frac \tau d}L^{(\tau)}_t=\sum_{x\in\eta_t} \xi_t(x,\eta_t)
$$
the conditions of Theorem 1.3 in \cite{EichelsbacherRaicSchreiber2013} are satisfied with $\alpha=1$ and $\beta=0$ there. This implies that there are constants $C_1,C_2,C_3>0$, such that
$$
\P( t^{\frac{\tau}d} |L_t^{(\tau)} - \E L_t^{(\tau)}|\geq x)
\leq
\exp{-\,\min\bigg\{ C_1 \frac{x^2}{t^{\frac{2\tau}d} \V L_t^{(\tau)}},\,C_2 x^{\frac 13},\, C_3 x^{\frac 34}t^{- \frac 14}\bigg\}}
$$
for all $x\geq 0$. Choosing $x=t^{\tau/d} u$ and applying Theorem~\ref{thm:AsymptoticCovariance} concludes the proof.
\end{proof}

\subsection{The binomial point process and the convex distance}
The investigations in this section were partially motivated by a recent application of the Gilbert graph to the study of so-called empty triangles in \cite{BMR}. In this context it was important to work with a binomial point process instead of a Poisson point process. For this reason we also include a section on concentration inequalities in the case of an underlying binomial point process (this was left as an open problem in \cite{EichelsbacherRaicSchreiber2013}).

Assume w.l.o.g.~that the observation window $W$ has volume one and let $(\d_n)_{n\in\N}$ be a sequence of positive real numbers such that $\d_n\to 0$, as $n\to\infty$. Fix an integer $n\geq 2$, let $X_1,\ldots,X_n$ be independent and uniformly distributed random points in $W$ and consider the random set $\xi= \{ X_1, \dots , X_n\}$. The Gilbert graph with vertex set $\xi$ has an edge between two points if their distance is at most $\d_n$. Similar to \eqref{def:Lalphat}, we define for $\tau\in\R$ the functional
$$ L^{(\tau)}_n(\xi) :=  \frac 12 \sum_{(X_i,X_j)\in\xi_{\neq}^2} \I(\| X_i - X_j\| \leq \d_n) \| X_i - X_j\|^{\tau}\,,$$
where the sum ranges over all pairs $(X_i,X_j)$ of distinct points of $\xi$. For these random variables we can deduce the following concentration inequalities.

\begin{theorem}\label{thm:LDIBinomial}
Let $\tau\geq 0$, let $m_n$ be the median of $L_n^{(\tau)}(\xi)$ and define
\begin{eqnarray*}
x^*(u,n)
&=&
\inf_{s>0}\frac{\ln n  + n\kappa_d \delta_n^d(e^s-1)}{2n\kappa_d \delta_n^ds}
\\ && +
\sqrt{ \frac{ u^2}{8n^2\kappa_d^2\delta_n^{2d+\tau}(u+m_n)s }+\bigg( \frac{\ln n + n \kappa_d \delta_n^d(e^s-1)}{2n\kappa_d \delta_n^ds} \bigg)^2} .
\end{eqnarray*}
Then, for $u>0$,
\begin{equation}\label{eq:LDIBinomial}
\P(|L_n^{(\tau)}(\xi)-m_n|\geq u) \leq 8 \exp{-\frac{ u^2}{8n\kappa_d\delta_n^{d+\tau}x^*(u,n)(u+m_n)}}
\end{equation}
and, in particular,
\begin{equation*}\label{eq:LDIBinomials1}
\P(|L_n^{(\tau)}(\xi)-m_n|\geq u) \leq
8 \exp{- \min\bigg\{\frac{u^2}{C_{n,W} (u+m_n)}, \frac {u}{ D_n \sqrt{(u+m_n)}}  \bigg\}}
\end{equation*}
with
$C_{n,W}= 16 (\ln(n)\delta_t^\tau+2n \kappa_d \delta_n^{d+\tau}) $
and $D_n= 4\sqrt{2 \d_n^\tau }$.

\end{theorem}

The proof of Theorem~\ref{thm:LDIBinomial} is very similar to that of Theorem~\ref{thm:LDIPoisson}, and for this reason we only give a short sketch of it. We need the  local version of $L_n^{(\tau)}(\xi)$, namely
$$ L^{(\tau)}_n(X_i;\xi) :=  \sum_{X_j\in\xi\setminus\{X_i\}} \I(\| X_i - X_j\| \leq \d_n) \| X_i - X_j\|^{\tau} \,, $$
and thus $L^{(\tau)}_n(\xi)  = \frac 12 \sum_{i=1}^n L^{(\tau)}_n(X_i;\xi)$.

Assume that an additional point set $\zeta= \{ y_1, \dots , y_n\}$ with $y_1,\ldots,y_n\in W$ is given. Then
\begin{equation*}\label{eqn:inequalityL}
\begin{split}
L^{(\tau)}_n (\xi)- L^{(\tau)}_n (\zeta)
&\leq
 \sum_{i=1}^n  L^{(\tau)}_n(X_i;\xi) \I(X_i \notin \zeta)
\leq
 \sum_{i=1}^n L^{(\tau)}_n(X_i;\xi) \I(X_i \neq y_i)\,.
\end{split}
\end{equation*}

To prove the deviation inequality, we use Talagrand's original convex distance.
For $\xi =\{X_1, \dots, X_n\}$ with $X_1,\hdots,X_n\in W$ and $A\subset W^n$ it is defined as
$$
d_T (\xi,A) := \max_{ u \in {\Sd}^{n-1}} \min_{ \zeta \in A}\, \sum_{i=1}^n u_i  \I(X_i \neq y_i)
$$
where, as usual, $u_i$, $i\in\{1,\hdots,n\}$, are the coordinates of $u \in {\Sd}^{n-1}$, cf.\ \cite[Definition 11.1]{DP}. In the following let $L_n^{(\tau)}(\xi)\neq 0$. We take the normalized vector of $(L^{(\tau)}_n(X_1;\xi), \dots ,  L^{(\tau)}_n(X_n;\xi))$ as $u$ and obtain 
\begin{eqnarray*}
d_T (\xi,A)&\geq& \nonumber
\min_{ \zeta \in A} \frac{1}{\sqrt{\sum_{i=1}^n L^{(\tau)}_n(X_i;\xi)^2}}
\Big(L^{(\tau)}_n (\xi) - L^{(\tau)}_n (\zeta) \Big) .
\end{eqnarray*}
If we assume that $L^{(\tau)}_n(X_i;\xi)$ is bounded by some $B>0$ for all $i\in\{1,\ldots,n\}$, we have
$ \sum_{i=1}^n L^{(\tau)}_n(X_i;\xi)^2  \leq 2 B L^{(\tau)}_n (\xi)$. Hence, it follows from the previous inequality that
$$
d_T (\xi,A) \geq \frac{1}{\sqrt {2 B } }\,
\min_{ \zeta \in A}  \frac{L^{(\tau)}_n (\xi) - L^{(\tau)}_n (\zeta)}{ \sqrt{L^{(\tau)}_n (\xi)}}
$$
if $L_n^{(\tau)}(\xi)\neq0$, $L_n^{(\tau)}(X_i,\xi)\leq B$ for all $i\in\{1,\hdots,n\}$ and $L_n^{(\tau)}(\xi)\geq L_n^{(\tau)}(\zeta)$ for all $\zeta\in A$. 
It was proved by Talagrand in \cite{Talagrand} that $d_T$ satisfies a large deviation inequality (see also Theorem 11.1 in \cite{DP}). Namely, for $A \subset W^n$ we have
\begin{equation*}
\P (A) \P  \left( d_T(\xi, A) \geq s  \right) \leq  \exp{- \frac {s^2} 4}  .
\end{equation*}

Precisely as in the Poisson case it follows that
\iffalse
Let us now make the relation between $d_T$ and $L^{(\tau)}_n(\xi)$  more explicit.
First, define
$ A = \{ \zeta\subset W^n:\ L^{(\tau)}_n (\zeta) \leq m_n \} $,
where $m_n$ is the median of $L^{(\tau)}_n(\xi)$. By definition of the median, we have $ \P(A) \geq \frac 12$.
Since the function $u\mapsto {u}/{\sqrt {u+ m_n}}$ is increasing, $ L^{(\tau)}_n(\xi) \geq u+ m_n$ and
$ d_T (\xi,A)   \geq
\frac{1}{\sqrt{2B \d_n^\tau}} \ \frac{L^{(\tau)}_n (\xi) -m_n}{\sqrt{L^{(\tau)}_n(\xi)}}
$
imply that
$d_T (\xi,A)   \geq
\frac{1}{\sqrt{2B \d_n^\tau}} \ \frac{u}{\sqrt{u+ m_n}}  $.
Together with \eqref{eq:Gegenereignis} and \eqref{eq:Talagrand}, this yields
\fi
$$
\P( L^{(\tau)}_n (\xi) \geq u+ m_n)
 \leq
2 \exp{\!- \frac{u^2} {8B (u+m_n) } }
+  \P( \exists i:\ L^{(\tau)}_n (X_i;\xi) > B)
$$
for $u>0$ and
$$
\P(L^{(\tau)}_n(\xi) \leq m_n-u)
\leq
2 \exp{- \frac {u^2}{8B m_n}}
+ 2  \P( \exists i:\ L^{(\tau)}_n (X_i;\xi) > B)
$$
for $m_n\neq 0$ and $u>0$.

Conditioned on $X_i$, $L^{(0)}_n(X_i;\xi)$ is a binomial random variable with distribution ${\rm Bin} (n-1,p(X_i))$ and
$$ p(X_i):= \P(\| X_i-X_1\| \leq \d_n \, | X_i) = V(B^d(X_i, \d_n)\cap W) \leq \kappa_d \d_n^d .
$$
For a binomial random variable $X\sim {\rm Bin}(m,p)$, the exponential Markov inequality implies the Chernoff bound
$$
\P(X\geq u) \leq \inf_{s\geq 0} \E \exp{s X- su}
\leq \inf_{s\geq 0} \exp{mp(e^{s}-1)- su}
$$
for $u>0$.
For $\tau\geq 0$ this yields
\begin{equation*}
\P( \exists i:\ L^{(\tau)}_n (X_i;\xi) > B) \leq n \inf_{s\geq 0} \exp{E(e^{s}-1)- s\delta_n^{-\tau} B}
\end{equation*}
with $E:=n\kappa_d\delta_n^d$. Combining these inequalities implies the following concentration inequality for $L_n^{(\tau)}(\xi)$:
\begin{proposition}\label{prop:BoundBinomial}
For $\tau\geq 0$ and $u>0$ we have
\begin{equation}\label{eq:BoundBinomial}
\begin{split}
\P( | L^{(\tau)}_n (\xi) -m_n |\geq u)
&\leq 
\inf_{s\geq 0, x>0}
\Big( 4 \exp{- \frac { u^2 }{ 8 n \kappa_d \d_n^{d+\tau} (u+m_n)x}}\\
& \hskip1.5cm
+  3 n \exp{n \kappa_d \delta_n^d(e^s-1- sx) }\Big)
\end{split}
\end{equation}
where $ m_n$ is the median of $L^{(\tau)}_n(\xi)$.
\end{proposition}

\begin{proof}[Proof of Theorem~\ref{thm:LDIBinomial}]
For fixed $s>0$, the right-hand side of \eqref{eq:BoundBinomial} becomes minimal -- up to a constant factor -- if both summands are equal. This implies
\begin{eqnarray*}
x&=&
\frac{\ln n/(n\kappa_d\delta_n^d)+e^s-1}{2s}
\\ &&
+\sqrt{\frac{ u^2}{8 n^2 \kappa_d^2  \delta_n^{2d+\tau}(u+m_n)s}+\bigg( \frac{\ln n/(n\kappa_d\delta_n^d)+e^s-1}{2s} \bigg)^2} .
\end{eqnarray*}
Choosing $s$ in such a way that $x$ becomes minimal leads to \eqref{eq:LDIBinomial}. For $s=1$ we obtain
$$
x \leq 2 \max\bigg\{  \frac{\ln n }{n\kappa_d\delta_n^d}+2 ,\, \sqrt{\frac{ u^2}{8 n^2 \kappa_d^2  \delta_n^{2d+\tau}(u+m_n)}}\bigg\}\,,
$$
which completes the proof.
\end{proof}

As an example how to combine both error terms in Proposition \ref{prop:BoundBinomial} we make our considerations more explicit in the sparse regime when $n^2\delta_n^d=c$, $c>0$. This has been of particular interest in \cite{BMR}.
From the previous computations leading to Proposition~\ref{prop:BoundBinomial} we can deduce the following corollary.

\begin{corollary}
Assume  $\tilde{B}>0$, $c >0$ and choose the sequence $(\d_n)_{n\in\N}$ such that $n^2 \d_n^d= c$ for all $n\in\N$. Then there is a constant $C=C(\tilde{B}, c,  \tau)>0$ such that
$$
 \P( L^{(\tau)}_n (\xi) \geq 9 \tilde{B}^2 \d_n^\tau \ln n) \leq C n^{-\tilde{B}+1} \,.
$$
\end{corollary}

\begin{proof}
Choosing $B=\tilde{B}\delta_n^{\tau}$ and $u=9\tilde{B}^2\delta_n^\tau \ln n - m_n$  we see that
\begin{eqnarray*}\label{eq:BoundCorrollaryBinomial}
\P(L_n^{(\tau)}(\xi)\geq u+m_n)
&\leq &
2 \exp{-\frac{(9\tilde{B}^2\delta_n^\tau \ln n - m_n)^2}{8\tilde{B} \delta_n^\tau 9\tilde{B}^2\delta_n^\tau \ln n}}
\\ && \hskip1cm
+ n \inf_{s\geq 0} \exp{E(e^s-1)-s\tilde{B}}\,.
\end{eqnarray*}
Since $m_n\leq 2 \E L_n^{(\tau)}(\xi)\leq \frac{d\kappa_d}{\tau+d} n^2 \delta_n^{d+\tau}= \frac{d\kappa_d}{\tau+d} c\,  \delta_n^\tau$, we have
$$
9\tilde{B}^2\delta_n^\tau \ln n - m_n \geq \sqrt{72} \tilde{B}^2\delta_n^\tau \ln n
$$
for sufficiently large $n$. Setting $s=\ln n$ we find
$$
n \inf_{s\geq 0} \exp{E(e^s-1)-s\tilde{B}} \leq \exp{E(n-1)} n^{-\tilde{B}+1}\,.
$$
Now, $n^2 \delta_n^d=c$ and $E=n\kappa_d\delta_n^d$ imply that $\exp{E(n-1)}$ is uniformly bounded in $n$. Combining the two previous inequalities concludes the proof.
\end{proof}


\bigskip
\begin{thebibliography}{30}

\bibitem{AzmoodehPeccati}
{\sc Azmoodeh, E.; Peccati, G.}: {\em Optimal Berry-Esseen bounds on the Poisson space}, arXiv: 1505.02578 (2015).

\bibitem{MateronConjecture}
{\sc Averkov, G.; Bianchi, G.}: {\em Confirmation of Matheron's conjecture on the covariogram of a planar convex body}, J. Eur. Math. Soc. \textbf{11}, 1187--1202 (2009).

\bibitem{BachmannPeccati}
{\sc Bachmann, S.; Peccati, G.}: {\em Concentration bounds for geometric Poisson functionals: logarithmic Sobolev inequalities revisited}, Electron. J. Probab. \textbf{21}, article 6 (2016).

\bibitem{BachmannReitzner}
{\sc Bachmann, S.; Reitzner, M.}: {\em Concentration for Poisson U-statistics: subgraph counts in random geometric graphs}, arXiv: 1504.07404 (2015).


\bibitem{BMR}
{\sc B\'ar\'any, I.; Marckert, J.-F.; Reitzner, M.}: {\em Many empty triangles have a common edge}, Discrete Comput. Geom. \textbf{50}, 244--252 (2013).

\bibitem{BourguinPeccati2012}
{\sc Bourguin, S.; Peccati, G.}: {\em Portmanteau inequalities on the Poisson space: mixed limits and multidimensional clustering,} Electron. J. Probab. \textbf{19}, article 66 (2014).

\bibitem{DST}
{\sc Decreusefond, L.; Schulte, M.; Th\"ale, C.}: {\em Functional Poisson approximation in Kantorovich-Rubinstein distance with applications to U-statistics and stochastic geometry}, Ann. Probab. \textbf{44}, 2147--2197 (2016).

\bibitem{DP}
{\sc Dubhashi, D.P.; Panconesi, A.}: {\em Concentration of Measure for the Analysis of Randomized Algorithms}, Cambridge University Press, Cambridge (2009).

\bibitem{EichelsbacherRaicSchreiber2013}
{\sc Eichelsbacher, P.; Rai\v{c}, M.; Schreiber, T.}: {\em Moderate deviations for stabilizing functionals in geometric probability}, Ann. Inst. H. Poincar\'e Probab. Statist \textbf{51}, 89--128 (2015).

\bibitem{Gal11}
{\sc Galerne, B.}: {\em Computation of the perimeter of measurable sets via their covariogram. Applications to random sets}, Image Anal. Stereol. \textbf{30}, 39--51 (2011).

\bibitem{Gilbert}
{\sc Gilbert, E. N.}: {\em Random plane networks}, J. Soc. Indust. Appl. Math. \textbf{9}, 533--543 (1961).

\bibitem{Kallenberg}
{\sc Kallenberg, O.}: {\em Foundations of Modern Probability}, Second Edition, Springer (2002).

\bibitem{Kingman}
{\sc Kingman, J.F.C.}: {\em Poisson Processes}, Oxford University Press (1993).

\bibitem{LachiezeReyPeccatiI}
{\sc Lachi\'eze-Rey, R; Peccati, G.}: {\em Fine Gaussian fluctuations on the Poisson space, I: contractions, cumulants and random geometric graphs}, Electron. J. Probab. \textbf{18}, article 32 (2013).

\bibitem{LachiezeReyPeccatiII}
{\sc Lachi\'eze-Rey, R; Peccati, G.}: {\em  Fine Gaussian fluctuations on the Poisson space, II: rescaled kernels, marked processes and geometric U-statistics}, Stochastic Process. Appl. \textbf{123}, 4186--4218 (2013).

\bibitem{Peccati2011}
{\sc Peccati, G.}: {\em The Chen-Stein method for Poisson functionals}, arXiv:1112.5051 (2011).

\bibitem{PSTU10}
{\sc Peccati, G.; Sol\'e, J. L.; Taqqu, M.S.; Utzet, F.}: {\em Stein's method and normal approximation of Poisson functionals}, Ann. Probab. {\bf 38}, 443--478 (2010).

\bibitem{PeccatiZheng10}
{\sc Peccati, G.; Zheng, C.}: {\em Multi-dimensional Gaussian fluctuations on the Poisson space}, Electron. J. Probab. {\bf 15}, 1487--1527 (2010).

\bibitem{Penrose03}
{\sc Penrose, M.D.}: {\em Random Geometric Graphs}, Oxford University Press, Oxford (2003).

\bibitem{Reitzner13}
{\sc Reitzner, M.}: {\em Poisson point processes: Large deviation inequalities for the convex distance}, Electon. Comm. Probab. \textbf{18}, article 96 (2013).

\bibitem{ReitznerSchulte2011}
{\sc Reitzner, M; Schulte, M.}: {\em Central limit theorems for U-statistics of Poisson point processes}, Ann. Probab. {\bf 41}, 3879--3909 (2013).

\bibitem{SW}
{\sc Schneider, R.; Weil, W.}: {\em Stochastic and Integral Geometry}, Springer, Berlin (2008).

\bibitem{SchulteDiss}
{\sc Schulte, M.}: {\em Malliavin-Stein Method in Stochastic Geometry}, PhD thesis, Universit\"at Osnabr\"uck, available at http://repositorium.uni-osnabrueck.de/handle/urn:nbn:de:gbv:700-2013031910717 (2013).

\bibitem{SchulteKolmogorov}
{\sc Schulte, M.}: {\em Normal approximation of Poisson functionals in Kolmogorov distance}, J. Theor. Probab. \textbf{29}, 96--117 (2016).

\bibitem{STScaling}
{\sc Schulte, M.; Th\"ale, C.}: {\em The scaling limit of Poisson-driven order statistics with applications in geometric probability}, Stoch. Proc. Appl. \textbf{122}, 4096-4120 (2012).

\bibitem{Sunklodas}
{\sc Sunklodas, J.}: {\em On a lower bound of the rate of convergence in the central limit theorem for $m$-dependent random fields}, Theor. Probab. Appl. \textbf{43}, 162--169 (1999).

\bibitem{Talagrand}
{\sc Talagrand, M.}: {\em Concentration of measure and isoperimetric inequalities in product spaces}, Publ. Math. IHES \textbf{81}, 73--205 (1995).

\end{thebibliography}
\end{document}